\documentclass[11pt]{amsart}
\usepackage[left=1in,right=1in,top=1in,bottom=1in]{geometry}

\usepackage[all,cmtip]{xy}
\usepackage{amssymb}
\usepackage[english]{babel}
\usepackage{mathrsfs}

\newtheorem{theorem}{Theorem}[section]
\newtheorem{lemma}[theorem]{Lemma}

\newtheorem{corollary}[theorem]{Corollary}

\theoremstyle{definition}
\newtheorem{definition}[theorem]{Definition}

\newtheorem{problem}[theorem]{Problem}

\newtheorem{remark}[theorem]{Remark}

\newtheorem{question}[theorem]{Question}

\newtheorem{fact}[theorem]{Fact}
\newtheorem{claim}{Claim}

\def\B{\mathcal{B}}

\def\N{\mathbb{N}}

\def\Z{\mathbb{Z}}

\def\Cyc{\mathrm{Cyc}}
\def\grp#1{\langle{#1}\rangle}
\def\ssgp{$\mathrm{SSGP}$}

\def\Z{\mathbb{Z}}

\def\P{\mathbb{P}}
\def\F{\mathbb{F}}

\def\ssgp{\mathrm{SSGP}}

\def\assgp{\mathrm{ASSGP}}

\def \Zet[#1]{\lceil #1 \rceil}  

\def\dash#1{\bar{#1}}

\def\wprod{*}
\def\fprod{\cdot}

\def\e#1{e}

\def\nice#1{\mathscr{#1}}

\def\lett{\mathrm{lett}}

\def\supp{\mathrm{supp}}

\def\wprod{*}

\def\fns{finite neighbourhood system}

\begin{document}
\title[SSGP topologies on free groups of infinite rank]
{SSGP topologies on free groups of infinite rank}

\author[D. Shakhmatov]{Dmitri Shakhmatov}
\address{Division of Mathematics, Physics and Earth Sciences\\
Graduate School of Science and Engineering\\
Ehime University, Matsuyama 790-8577, Japan}
\email{dmitri.shakhmatov@ehime-u.ac.jp}

\author[V. Ya\~nez]{V\'{\i}ctor Hugo Ya\~nez}
\address{Doctor's Course, Graduate School of Science and Engineering\\
Ehime University, 
Matsuyama 790-8577, Japan}
\email{victor\textunderscore yanez@comunidad.unam.mx}
\thanks{This paper was written as part of the second listed author's Doctor's Program at the Graduate School of Science and Engineering of Ehime University. The second listed author was partially supported by the Matsuyama Saibikai Grant.}

\begin{abstract}
We prove that every free group $G$ with infinitely many generators admits a
Hausdorff group topology $\mathscr{T}$ with the following property: for every $\mathscr{T}$-open neighbourhood $U$ of the identity of $G$,
each element 
$g\in G$ can be represented as a product
$g=g_1 g_2\dots g_k$ such that the cyclic group generated by each $g_i$ is contained in $U$.
In particular, $G$ admits a Hausdorff group topology with the small subgroup generating property of Gould.
This provides a positive answer to a question of Comfort and Gould
in the case of free groups with infinitely many generators.
The case of free groups with finitely many generators remains open.
\end{abstract}

\dedicatory{In memory of W. Wistar Comfort}

\maketitle

\section{introduction}

As usual, $\N$ denotes the set of natural numbers and we let $\N^+=\N\setminus\{0\}$. 

Let $G$ be a group.
For subsets $A,B$ of 
$G$, we let 
$$
AB=\{ab:a\in A, b\in B\}
\
\text{ and }
\
A^{-1}=\{a^{-1}:a\in A\}. 
$$
We say that a subset $A$ of $G$ is \emph{symmetric} if and only if $A = A^{-1}$.
For a subset $A$ of 
$G$, we
denote by $\grp{A}$ the smallest subgroup of $G$ containing $A$.
To simplify the notation, 
we write $\grp{x}$ instead of $\grp{\{x\}}$ for $x\in G$.

A topological group is \emph{minimally almost periodic} \cite{NW} if every continuous homomorphism from it to a compact group is trivial. We refer the reader to \cite{G, CG, DS_SSGP, DS_MinAP} for a historical overview of examples in this class of groups.
Answering a long-standing question of Comfort and Protasov,
Dikranjan and the first author gave a complete characterization of abelian groups which admit an introduction of a minimally almost periodic group topology \cite{DS_MinAP}.

Following the notation from \cite{DS_SSGP}, we define 
\begin{equation}
\label{eq:Cyc}
\Cyc(A)=\{x\in G:
\grp{x}\subseteq A\}
\ 
\text{ for every }
A\subseteq G.
\end{equation}

\begin{definition}
\label{SSGP:original}
A topological group $G$ has the {\em small subgroup generating property\/} (abbreviated to {\em SSGP\/})
if and only if 
$\grp{\Cyc(U)}$ is dense in $G$ for every neighbourhood $U$ of the identity of 
$G$.
We shall say that a topological group $G$ is {\em SSGP\/} if $G$ satisfies the small subgroup generating property.
\end{definition}

The small subgroup generating property was defined by Gould in \cite{G}. Examples of $\ssgp$ groups can be found in \cite{D-W, DS_SSGP, G, G-Zomega, Countable_SSGP, SY-assgp}. 

Comfort and Gould \cite{CG} asked the following question.

\begin{question}\label{ques1a} 
\cite[Question 5.2]{CG} What are the (abelian) groups which admit an $\ssgp$ group topology?
\end{question}

An ``almost complete'' characterization of \emph{abelian} topological groups which admit an $\ssgp$ group topology was 
obtained in \cite{DS_SSGP}, with the remaining case 
resolved in \cite{Countable_SSGP}. 

Question \ref{ques1a}  remains widely open for non-abelian groups.
If a set $X$ has at least two elements, then its symmetric group $S(X)$ does not admit an $\ssgp$ group topology \cite[Example 5.4(c)]{DS_SSGP}. 
In this paper we essentially resolve Question \ref{ques1a} for free groups.

\section{Results}

The $\ssgp$ property was studied as a generalization of a stronger property utilized by Dierolf and Warken in \cite{D-W} as a means to 
prove that the Hartman-Mycielski group \cite{HM} is 
minimally almost periodic. 
In the following definition we 
propose
a name for this property, as well as state it using the same terminology as in Definition \ref{SSGP:original}.
\begin{definition}
\label{def:ASSGP}
A topological group $G$ has the {\em algebraic small subgroup generating property\/} (abbreviated to {\em ASSGP\/})
if and only if the equality 
$G = \grp{\Cyc(U)}$ holds for every neighbourhood $U$ of the identity of 
$G$.
We shall say that a topological group $G$ is {\em ASSGP\/} if $G$ satisfies the algebraic small subgroup generating property.
\end{definition}

It is clear from Definitions \ref{SSGP:original} and \ref{def:ASSGP} that ASSGP implies SSGP.

The main goal of this paper is to prove the following two theorems:

\begin{theorem}
\label{countable:theorem}
The free group $F(X)$ over 
a countably infinite set 
$X$ admits a metric $\assgp$ group topology.
 \end{theorem}

\begin{theorem}
\label{the:theorem}
Every free group 
with infinitely many generators
admits 
an $\assgp$ group topology.
\end{theorem}

The proofs of these two theorems are postponed until Sections
\ref{Sec:7} and \ref{sec:8}, respectively.

The paper is organized as follows. Basic facts about free groups are recalled in Section~\ref{free:section}. In Section~\ref{sec:nghbd} we introduce a notion of a \fns\ on a free group; this is basically a finite initial segment of a countable family of future neighbourhoods in some group topology on this group. 
In Section~\ref{sec:extensions}, a notion of an extension of a \fns\ is defined; this is a \fns\ on a bigger free group whose traces of new neighbourhoods to the smaller free group coincide with the original neighborhoods.
In Section~\ref{sec:enrichments}, we devise a technique for extending a \fns\ to a \fns\ on a bigger group, and provide canonical representations of elements of ``extended neighbourhoods'' by elements from ``smaller neighbourhoods''
and a fixed set which can be viewed as a base for such an extension.
Section~\ref{sec:aux} contains three auxiliary lemmas, the main of which is Lemma~\ref{main:lemma:assgp} responsible for the $\ssgp$ property of the topology under construction.
In Section~\ref{Sec:6}, we introduce a partially ordered set which is used in the proof of Theorem~\ref{countable:theorem} (the countable case); the proof itself is carried out in Section~\ref{Sec:7}. Theorem~\ref{the:theorem} (the general case) is proved in Section~\ref{sec:8}. Its proof simply provides a reduction of general case to the countable case.
Finally, open questions are listed in Section~\ref{sec:questions}.

In the proof of Theorem \ref{countable:theorem}, we use a partially ordered set to produce a topology on the free group with a countably infinite set of generators. This technique was used by the authors recently in \cite{Axioms} and \cite{Countable_SSGP}.

Theorems \ref{countable:theorem} and \ref{the:theorem}
 were announced by the authors in 
\cite{SY-kyoto1}.

\section{The 
free group $F(X)$ over a set $X$}
\label{free:section}

\begin{definition}
\label{def:W(X)}
Let $X$ be a 
set. 
\begin{itemize}
\item[(i)]
Let $W_0 = \{ \emptyset\}$. For $n\in\N^+$, let
$$
W_n = \{x_1^{\varepsilon_1} \dots x_n^{\varepsilon_n}: x_i \in X, \varepsilon_i \in \{-1,1\} \text{ for all } i = 1, \dots, n \}.
$$
\item[(ii)] For $w=x_1^{\varepsilon_1} \dots x_n^{\varepsilon_n}\in W_n(X)$
and $v=y_1^{\delta_1} \dots y_m^{\delta_m}\in W_m(X)$, we let 
$w=v$ if and only if $n=m$ and $x_i=y_i$, $\varepsilon_i=\delta_i$ for all $i = 1,\dots,n$.
\item[(iii)]
Elements of the set 
\begin{equation*}
W(X) = \bigcup_{n \in \N}W_n(X)
\end{equation*}
are called \emph{words in alphabet $X$}.
According to (ii), this union consists of pairwise disjoint sets, so
for every word $w\in W(X)$, there exists a unique $n\in\N$ such that $w\in W_n(X)$; this $n$ is called the {\em length\/} of $w$ and denoted by $l(w)$.
\item[(iv)]
Given a word $w = x_1^{\varepsilon_1} \dots x_n^{\varepsilon_n} \in W_n(X)$, a {\em sub-word of $w$\/} is a word $w' = x_k^{\varepsilon_k} \dots x_l^{\varepsilon_l}$ for some $k,l \in \N$ such that $1 \le k \le l \le n$. 
The word $w'$ is said to be an {\em initial sub-word\/} of $w$
when $k=1$ and a {\em final sub-word\/} of $w$
when $l=n$.
\end{itemize}
\end{definition}

The \emph{empty word\/} $\emptyset$ will be denoted also by $\e{W}$. Clearly, we have that $l(\e{W})=0$.

We can define an operation on the set $W(X)$ as follows:

\begin{definition} \label{def:conc}
Let $X$ be a 
set. \begin{itemize}
\item[(i)] For $w=x_1^{\varepsilon_1} \dots x_n^{\varepsilon_n}\in W_n(X)$
and $v=y_1^{\delta_1} \dots y_m^{\delta_m}\in W_m(X)$, 
the word $x_1^{\varepsilon_1} \dots x_n^{\varepsilon_n} y_1^{\delta_1} \dots y_m^{\delta_m}\in W_{n+m}(X)$ is called the (result of) {\em concatenation of $w$ and $v$\/}; we denote this word by $w\wprod v$. We also let $\e{W}\wprod w=w\wprod\e{W}=w$ for every word $w\in W(X)$.
\item[(ii)] For $w=x_1^{\varepsilon_1} \dots x_n^{\varepsilon_n}\in W_n(X)$,  the word $w^{-1} = x_n^{- \varepsilon_n} \dots x_1^{- \varepsilon_1}\in W_n(X)$
is called the {\em inverse of $w$\/}.
We also let $\e{W}^{-1}=\e{W}$.
\end{itemize}
\end{definition}

The set $W(X)$ equipped with the binary operation $\wprod$ is a semigroup with the identity $\e{W}$.

The proof of the following lemma is straightforward.

\begin{lemma}
\label{initial:final:subwords}
For every initial sub-word $w'$ of a word $w\in W(X)$, there exists a unique final sub-word $w''$ of $w$ such that $w=w'\wprod w''$. 
Conversely,
for every final sub-word $w''$ of $w$, there exists an initial sub-word $w'$ of $w$ such that $w=w'\wprod w''$. 
\end{lemma}

\begin{definition} \label{def:f(x)}
\label{def:F(X)}
Let $X$ be a 
set. A word $w=x_1^{\varepsilon_1} \dots x_n^{\varepsilon_n}\in W(X)$ is  \emph{irreducible\/} provided that, for every $i \in \{1, \dots, n-1\}$, either 
$x_i \neq x_{i+1}$ or ${\varepsilon_i} = {\varepsilon_{i+1}}$. Observe that the empty word $\e{W}$ is considered irreducible too. We shall denote by $F(X)$ the set of all irreducible words $w\in W(X)$. 
\end{definition}

\begin{lemma}
\label{cancellation:of:two:words}
For every pair of irreducible words $v,w\in F(X)$, there exist unique $v',v'',w',w''\in W(X)$
such that
$v=v'\wprod v''$,  $w=w' \wprod w''$, $v''$ and $w'$ are inverses of each other and
$v'\wprod w''$ is an irreducible word.
\end{lemma}
\begin{proof}
Let $v''$ be the final sub-word of $v$ of maximal length $l(v'')$ such that its inverse $w'=(v'')^{-1}$ is an initial sub-word of $w$.
Use Lemma \ref{initial:final:subwords} to find a unique initial sub-word
$v'$ of $v$ and a unique final sub-word $w''$ of $w$ such that
$v=v'\wprod v''$ and $w=w'\wprod w''$.
Finally, note that the word $v'\wprod w''$ is irreducible by the maximality of $v''$ and Definition 
\ref{def:f(x)}.
\end{proof}

\begin{definition}
\label{def:product:in:F(X)}
For a 
set $X$, we define a binary operation $\fprod$ on the set $F(X)$ as follows.
For $v,w\in F(X)$, let $v',v'',w',w''\in W(X)$ be the unique words 
as in the conclusion of Lemma
\ref{cancellation:of:two:words}.
Then we define $v\fprod w= v' \wprod w''$.
\end{definition}

The following fact is well-known.

\begin{fact}
The $\cdot$ operation on $F(X)$ is associative.
\end{fact}

From this fact, Lemma \ref{cancellation:of:two:words} and observing that $e$ behaves as a neutral element, we obtain that $F(X)$ equipped with the operation $\cdot$ is a group:

\begin{lemma}
\label{lem:F(X)}
For every 
set $X$, the set $F(X)$ equipped with the binary operation $\fprod$ is a group with the identity 
$\e{F}$.
The inverse of an element $w\in F(X)$ in $F(X)$ coincides
with the (irreducible) word $w^{-1}$ defined in item (ii) of Definition \ref{def:conc}.
\end{lemma}

\begin{definition}
The group $F(X)$ from Lemma \ref{lem:F(X)} is called the {\em free group over $X$\/}.
\end{definition}

Let us observe the following fundamental property of the free group:

\begin{lemma}
Let $X$ be a 
set and $G$ be any group. Every mapping $f: X \to G$ has an extension to an homomorphism $\hat{f}: F(X) \to G$ such that $\hat{f} \restriction_X = f$.
\end{lemma}

\begin{proof}
Let $X,f$ and $G$ be as in the hypotheses. Given a non-empty word $w \in F(X)$ there exist unique $x_1, \dots, x_n \in X$ and $\varepsilon_1, \dots, \varepsilon_n \in \{-1,1\}$ such that $w = x_1^{\varepsilon_1}, \dots, x_n^{\varepsilon_n}.$ Define the mapping $\hat{f}$ such that $\hat{f}(w) = f(x_1)^{\varepsilon_1} \cdots f(x_n)^{\varepsilon_n}$.

Clearly, this mapping defines an homomorphism and furthermore $\hat{f} \restriction_X = f$ is satisfied by construction.
\end{proof}

\begin{definition}
\label{def:letters}
Let $X$ be a 
set. For a word $w=x_1^{\varepsilon_1} \dots x_n^{\varepsilon_n}\in W(X)$ with $n\ge 1$, 
$$
\lett(w)=\{x_1,\dots,x_n\}
$$
denotes the set of all letters $x_i$ appearing in $w$. We also let 
$\lett(\e{W})=\lett(\emptyset)=\emptyset$.
\end{definition}

The set $\lett(w)$ from the above definition coincides with the support  of the word $w$ in the variety of all groups  \cite{DS_Fundamenta}. It is worth noticing that the notion of the support of an element of a free group was  
introduced in \cite{DS_Fundamenta} for arbitrary varieties of groups.

We finish this section with two lemmas which shall be needed in the future.

\begin{lemma}
\label{letters:and:products}
\begin{itemize}
\item[(i)]
$\lett(v \cdot w) \subseteq \lett(v) \cup \lett(w)$ for all $v,w \in F(X)$.
\item[(ii)]
If $a_1,a_2,\dots,a_m\in F(X)$ and $h=a_1\fprod a_2\fprod \dots\fprod a_m$,
then
$\lett(h)\subseteq \bigcup_{l=1}^m \lett(a_l)$.
\end{itemize}
\end{lemma}
\begin{proof}
(i)
Let $v', v'', w', w''\in F(X)$ be as in Definition \ref{def:product:in:F(X)}. 
Then $v\fprod w=v'\wprod w''$, which implies
$\lett(v\fprod w)=\lett(v'\wprod w'')
=\lett(v')\cup\lett(w'')$
 by Definition \ref{def:letters}.
Since $v'$ is a sub-word of $v$, we have 
$\lett(v')\subseteq \lett(v)$ by Definition \ref{def:letters}.
Similarly, since $w''$ is a sub-word of $w$, we have 
$\lett(w'')\subseteq \lett(w)$. This proves item (i).

Item (ii) is proved by induction making use of item (i).
\end{proof}

\begin{lemma}
\label{support:lemma}
Suppose that $x_1,x_2,\dots,x_s\in X$,
$\varepsilon_1,\varepsilon_2,\dots,\varepsilon_s\in\{-1,1\}$,
$g_0=x_1^{\varepsilon_1}x_2^{\varepsilon_2}\dots x_s^{\varepsilon_s}\in F(X)$, $j=1,\dots,s$ and 
$x_j\not=x_p$ for $p=1,\dots,s$ with $j\not=p$.
Then $x_j\in\lett(h)$ for every $h\in \grp{g_0}\setminus\{e\}$. 
\end{lemma}
\begin{proof}
Fix $h\in\grp{g_0}$. Then $h=g_0^q$ for some integer $q$. 
Consider the map $f: X\to F(X)$ satisfying 
\begin{equation}
\label{special:f}
f(x_j)=x_j
\ 
\text{ and }
\ 
f(x)=e
\text{ for }
x\in X\setminus\{x_j\}.
\end{equation}
Let $\hat{f}: F(X) \to F(X)$ be the homomorphism such that $\hat{f} \restriction_X = f$.
Since $\hat{f}$ is a homomorphism extending $f$ and \eqref{special:f} holds, we have
$\hat{f}(h)=\hat{f}(g_0^q)=x_j^q$.

Suppose that $x_j\not\in \lett(h)$. Then $\lett(h)\subseteq X\setminus\{x_j\}$, and so
$\hat{f}(h)=e$
by \eqref{special:f}.
This shows that $x_j^q=e$, which implies $q=0$. Thus, $h=g_0^q=e$.
\end{proof}

\section{\fns s}
\label{Sec:Trees}
\label{sec:nghbd}

\begin{definition}
\label{def:bar}
Let $X$ be a set.
Given any subset $A \subseteq F(X)$, we define $A^{-1} = \{a^{-1}: a \in A\}$, and $\dash{A} = A \cup A^{-1} \cup \{e\}$.
\end{definition}

\begin{definition}
\label{SPSa}
Let $X$ be a 
set.
A {\em \fns\ of $F(X)$\/} is a 
finite
sequence
$\nice{U} = \{U_i: i \leq n\}$ (where $n\in\N^+$) 
satisfying 
the following conditions:
\begin{itemize}
\item[(1$_{\nice{U}}$)] 
$U_i\subseteq F(X)$ for every $i\le n$,
\item[(2$_\nice{U}$)] $U_{i}^{-1} = U_i$ for every 
$i\leq n$,
\item[(3$_\nice{U}$)] $ \bigcup_{x \in \dash{X}} x \fprod U_{i+1} \fprod U_{i+1} \fprod x^{-1} \subseteq U_i$ for every 
$i < n$,
\item[(4$_\nice{U}$)] 
$e\in U_n$.
\end{itemize}
\end{definition}

\begin{remark}
\label{identity:remark}
{\em If $X$ is a set and $\nice{U} = \{U_i: i \leq n\}$ is a \fns\ for $F(X)$, then 
$e\in U_i$ for every $i\le n$.\/}
This statement is proved by finite reverse induction on $i=n, n-1,\dots, 0$.
For $i=n$, the statement holds by (4$_\nice{U}$).
Suppose now that $i<n$ and we have already proved that $e\in U_{i+1}$. Since $e\in\bar{X}$ by Definition \ref{def:bar}, 
$e=e\fprod e\fprod e\fprod e\in \bigcup_{x \in \dash{X}} x \fprod U_{i+1} \fprod U_{i+1} \fprod x^{-1} \subseteq U_i$ by (3$_\nice{U}$).
\end{remark}

\begin{definition}
\label{def:enrichment}
Let $\nice{U} = \{U_i : i \leq n\}$ be a finite sequence of subsets of $F(X)$ for some set $X$.
\begin{itemize}
\item[(i)]
Let $B\subseteq F(X)$. Define 
\begin{equation}
\label{eq:def:Vn}
V_n=U_n\cup B.
\end{equation}
By finite reverse induction on $i=n-1,n-2,\dots,0$, define
\begin{equation}
\label{eq:def:Vi}
V_i=U_i\cup\bigcup_{x\in \dash{X}} x\fprod V_{i+1}\fprod V_{i+1}\fprod x^{-1}.
\end{equation}
We shall call the sequence $\nice{V}=\{V_i:i\le n\}$ the {\em $B$-enrichment of the sequence $\nice{U}$ in $F(X)$\/}.
\item[(ii)]
For a set $C\subseteq F(X)$, we shall call the $\left(\bigcup_{c\in C} \grp{c}\right)$-enrichment of 
$\nice{U}$ in $F(X)$ the {\em cyclic $C$-enrichment of $\nice{U}$ in $F(X)$\/}.
\end{itemize}
\end{definition}

\begin{lemma}
\label{enrichment:lemma}
Let $X$ be a set and $\nice{U}=\{U_i:i\le n\}$ be a finite sequence
such that:
\begin{itemize}
\item[(a)] $U_i\subseteq F(X)$ for all $i\le n$,
\item[(b)] $U_i^{-1}=U_i$ for every $i\le n$,
\item[(c)] $e\in U_n$. 
\end{itemize}
Furthermore, 
let $B\subseteq F(X)$ be a set satisfying
\begin{itemize}
\item[(d)] $B^{-1}=B$.
\end{itemize}
Then the $B$-enrichment of $\nice{U}$ in $F(X)$ is a \fns\ for $F(X)$.
\end{lemma}

\begin{proof}
Let $\nice{V}=\{V_i:i\le n\}$ be the $B$-enrichment of the sequence $\nice{U}$ in $F(X)$. It suffices to check conditions
(1$_\nice{V}$)--(4$_\nice{V}$) of Definition \ref{SPSa}.

(1$_\nice{V}$) Since $U_n$ and $B$ are subsets of $F(X)$ by our assumption, $V_n\subseteq F(X)$ by \eqref{eq:def:Vn}.
Note that $\dash{X}\subseteq F(X)$.
Applying finite reverse induction on $i=n-1,n-2,\dots,0$,
one concludes from this, \eqref{eq:def:Vi}
and $\bigcup_{i\le n} U_i\subseteq F(X)$ that $V_i\subseteq F(X)$ for all $i\le n$.

(2$_\nice{V}$) 
We shall prove by finite reverse induction on $i=n,n-1,\dots,0$
that $V_i^{-1}=V_i$.
First, note that
$V_n^{-1}=U_n^{-1}\cup B^{-1}=U_n\cup B=V_n$ by \eqref{eq:def:Vn}, (b) and (d).
Assume now that $i<n$ and the equation $V_{i+1}^{-1}=V_{i+1}$ has already been proved.
It easily follows from this inductive assumption 
that
\begin{equation}
\label{conj:are:symmetric}
\left(\bigcup_{x\in \dash{X}} x\fprod V_{i+1}\fprod V_{i+1}\fprod x^{-1}\right)^{-1}
=
\bigcup_{x\in \dash{X}} x\fprod V_{i+1}^{-1}\fprod V_{i+1}^{-1}\fprod x^{-1}
=
\bigcup_{x\in \dash{X}} x\fprod V_{i+1}\fprod V_{i+1}\fprod x^{-1}.
\end{equation}
Since $U_i^{-1}=U_i$ by (b), from \eqref{eq:def:Vi} and \eqref{conj:are:symmetric} we conclude that $V_i^{-1}=V_i$.

(3$_\nice{V}$)   is straightforward from \eqref{eq:def:Vi}.

(4$_\nice{V}$) is straightforward from (c) and \eqref{eq:def:Vn}.
\end{proof}

\begin{remark}
\label{nghb:systems:remark}
{\em Every \fns\ $\nice{U}=\{U_i:i\le n\}$ for $F(X)$ satisfies the assumptions of Lemma \ref{enrichment:lemma}.}
Indeed, item (a) follows from (1$_\nice{U}$), 
item (b) follows from (2$_\nice{U}$), and item (c) follows from (4$_\nice{U}$).
\end{remark}

\begin{corollary}
\label{cor:1}
For every symmetric subset $B$ of $F(X)$ and each \fns\
$\nice{U}=\{U_i:i\le n\}$ for $F(X)$, the $B$-enrichment of $\nice{U}$ in $F(X)$
is a \fns\ for $F(X)$.
\end{corollary}
\begin{proof}
By Remark \ref{nghb:systems:remark}, $\nice{U}$ satisfies items (a)--(c) of Lemma \ref{enrichment:lemma}.
Item (d) of this lemma holds because $B$ is symmetric by our assumption.
Now the conclusion of our corollary follows from 
the conclusion of Lemma \ref{enrichment:lemma}.
\end{proof}

\begin{definition} 
\label{extension:def}
Let $X$ and $Y$ be sets such that $X\not=\emptyset$ and $X\subseteq Y$. For a \fns\ $\nice{U}$ for $F(X)$, we shall denote by $\nice{U}_Y$ the 
cyclic $(Y\setminus X)$-enrichment
of $\nice{U}$ in $F(Y)$.
\end{definition}

\begin{corollary}
\label{cor:2}
Let $X$ and $Y$ be sets such that $X\not=\emptyset$ and $X\subseteq Y$. For each \fns\ $\nice{U}$ for $F(X)$, its
cyclic $(Y\setminus X)$-enrichment
$\nice{U}_Y$ is a \fns\ for $F(Y)$.
\end{corollary}

\begin{proof}
Since $F(X)\subseteq F(Y)$, it follows from 
Remark \ref{nghb:systems:remark} that $\nice{U}$ satisfies items (a)--(c) of Lemma \ref{enrichment:lemma} (with $X$ replaced by $Y$). Since $B=\bigcup_{y\in Y\setminus X} \grp{y}$ is a symmetric subset of $F(Y)$, item (d) of Lemma \ref{enrichment:lemma} is satisfied as well. Applying this lemma, we 
conclude that the
$B$-enrichment
of $\nice{U}$ is a \fns\ for $F(Y)$. It remains only to note that
this $B$-enrichment coincides with $\nice{U}_Y$ by Definitions 
\ref{def:enrichment}(ii) and \ref{extension:def}.
\end{proof}

\section{Extension of \fns s}
\label{sec:extensions}

\begin{definition}
\label{SPSb}
Given two 
sets $X$ and $Y$,
we shall say that a \fns\ $\nice{V} = \{V_i:i\leq m\}$ for $F(Y)$ is
an \emph{extension\/} of a \fns\ $\nice{U} = \{U_i: i \leq n\}$
for $F(X)$
if and only if the following conditions are satisfied:
\begin{itemize}
\item[(i$_{\nice{V}}^{\nice{U}}$)] $X \subseteq Y$, so $F(X)\subseteq F(Y)$,
\item[(ii$_{\nice{V}}^{\nice{U}}$)] $n \le m$,
\item[(iii$_{\nice{V}}^{\nice{U}}$)] $V_i\cap F(X) =U_i$ for every $i\leq n$.
\end{itemize}
\end{definition}

A straightforward proof of the next lemma is left to the reader.

\begin{lemma}
\label{trivial-extension:by:e}
Let $X$ be a 
set, and 
let $\nice{U} = \{U_i: i \leq n\}$ be a \fns\ for  $F(X)$.
Let $m\in\N$ such that $m>n$.
Define $V_i=U_i$ for $i\le n$ and 
$V_i=\{e\}$ for $n<i\le m$.
Then
$\nice{V}=\{V_i:i\le m\}$ is a \fns\ for $F(X)$ extending $\nice{U}$. 
\end{lemma}

\begin{lemma}
\label{cyclic:C-enrichment}
Let $X$ be a 
set and $\nice{U}$ be a
\fns\  for $F(X)$.
Then for every set $Y$ containing $X$ and each 
set $C\subseteq \grp{Y\setminus X}$,
the cyclic $C$-enrichment $\nice{V}$
of $\nice{U}$ extends it.
\end{lemma}

\begin{proof}
Let $\nice{U}=\{U_i:i\le n\}$ be a \fns\ for $F(X)$ and let 
$\nice{V}=\{V_i:i\le n\}$ be its 
cyclic $C$-enrichment 
in $F(X)$.
By Definition \ref{def:enrichment}, we have
\begin{equation}
\label{eq:def:Vn:Y}
V_n=U_n\cup \bigcup_{c\in C} \grp{c}
\end{equation}
and
\begin{equation}
\label{eq:def:Vi:Y}
V_i=U_i\cup\bigcup_{y\in \dash{Y}} y \fprod V_{i+1}\fprod V_{i+1}\fprod y^{-1}
\end{equation}
for $i=0,\dots,n-1$.

Let $p_X:Y\to F(X)$ be the map which sends each $y\in Y\setminus X$ to $e$ and is the identity on $X$, and let 
$\pi_X: F(Y) \to F(X)$ be the  homomorphism which extends it.
Then 
\begin{equation}
\label{properties:pi_X}
\pi_X\restriction_{F(X)}=\mathrm{id}_{F(X)}
\ 
\text{ and }
\ 
\pi_X[\grp{Y\setminus X}]=\{e\},
\end{equation}
where $\mathrm{id}_{F(X)}$ denotes the identity map of $F(X)$.
Since $U_i\subseteq F(X)$ by (1$_{\nice{U}}$),  the first equation in 
\eqref{properties:pi_X} implies that
\begin{equation}
\label{Ui:are:not:moved}
\pi_X(U_i)=U_i
\text{ for every }
i=0,1,\dots,n.
\end{equation}
\begin{claim}
\label{claim:1}
$\pi_X[V_i] \subseteq U_i$  for every $i \leq n$.
\end{claim}

\begin{proof}
We shall prove our claim by reverse induction on $i=n,n-1,\dots,0$.

Since 
$\pi_X$ is a homomorphism and $C\subseteq \grp{Y\setminus X}$,
from \eqref{eq:def:Vn:Y}, \eqref{properties:pi_X}, \eqref{Ui:are:not:moved} and (4$_\nice{U}$),
we get
\begin{equation*}
\pi_X[V_n] = \pi_X\left[U_n \cup \bigcup_{c \in C} \grp{y}\right] = \pi_X[U_n]\cup\{e\} = U_n\cup\{e\}=U_n.
\end{equation*} 

Suppose that $i<n$ and the inclusion $\pi_X[V_{i+1}] \subseteq U_{i+1}$ has already been proved.
We shall show that $\pi_X[V_i] \subseteq U_i$. 

Let $y \in \dash{Y}$ and $v_1,v_2 \in V_{i+1}$ be arbitrary. 
By inductive hypothesis we have that $\pi_X(v_k) \in \pi_X[V_{i+1}]\subseteq U_{i+1}$ for $k=1,2$. Note that $\pi_X(y) \in \dash{X}$ by
the definition of the homomorphism $\pi_X$ and Definition \ref{def:bar}.
Therefore,
\begin{equation*}
\pi_X(y \fprod v_1 \fprod v_2 \fprod y^{-1}) = \pi_X(y) \cdot  \pi_X(v_1) \cdot \pi_X(v_2) \cdot \pi_X(y)^{-1}\in \bigcup_{x \in \dash{X}} x \cdot U_{i+1} \cdot U_{i+1} \cdot x^{-1} \subseteq U_i
\end{equation*}
by (3$_\nice{U}$). This shows that
\begin{equation}
\label{eq:9:d}
\pi_X\left[\bigcup_{y\in \dash{Y}} y\fprod V_{i+1}\fprod V_{i+1}\fprod y^{-1}\right]
\subseteq U_i.
\end{equation}

Since $\pi_X$ is a homomorphism, 
combining \eqref{eq:def:Vi:Y}, \eqref{Ui:are:not:moved}, \eqref{eq:9:d}, we obtain
$$
\pi_X[V_i] = 
\pi_X\left[U_i\cup\bigcup_{y\in \dash{Y}} y\fprod V_{i+1}\fprod V_{i+1}\fprod y^{-1}\right]
=
\pi_X[U_i] \cup \pi_X\left[ \bigcup_{y \in \dash{Y}} y \fprod V_{i+1} \fprod V_{i+1} \fprod y^{-1}\right]\subseteq U_i\cup U_i=U_i.
$$
This finishes the inductive step.
\end{proof}

\begin{claim}
$V_i \cap F(X) = U_i$ for every $i \leq n$.
\end{claim}

\begin{proof}
The inclusion $U_i\subseteq V_i \cap F(X)$ follows from 
\eqref{eq:def:Vn:Y}, \eqref{eq:def:Vi:Y}
and (1$_{\nice{U}}$). To show the inverse inclusion
$V_i \cap F(X) \subseteq U_i$,
let $h\in V_i \cap F(X)$ be arbitrary. Then $h=\pi_X(h)\in \pi_X[V_i]\subseteq U_i$ by the first equation in \eqref{properties:pi_X} and Claim \ref{claim:1}.
\end{proof}
Since $n \leq n$ and 
$X\subseteq Y$,
conditions (i$_{\nice{U}_Y}^{\nice{U}} $) and (ii$_{\nice{U}_Y}^{\nice{U}} $) of Definition \ref{SPSb} are satisfied.
Condition (iii$_{\nice{U}_Y}^{\nice{U}} $) holds by the previous claim.
Since all conditions from Definition \ref{SPSb} are met,  $\nice{V}$ extends $\nice{U}$.
\end{proof}

\begin{corollary}
\label{extension:set:b}
If $X$ and $Y$ are sets such that $X\not=\emptyset$ and $X\subseteq Y$, then for every \fns\ $\nice{U}$ for $F(X)$, its cyclic 
$(Y\setminus X)$-enrichment $\nice{U}_Y$ 
extends $\nice{U}$.
\end{corollary}

\section{Canonical representations for elements of neighbourhoods of $B$-enrichments}
\label{sec:enrichments}

\begin{definition}
\label{canonical:representation:definition}
Assume that $\nice{U} = \{U_i : i \leq n\}$ is a finite sequence of subsets of $F(X)$, $B\subseteq F(X)$ and $\nice{V}=\{V_i:i\le n\}$ is the $B$-enrichment of $\nice{U}$ in $F(X)$. 
By
finite reverse induction 
on $i=n,n-1,\dots,0$, we shall define a (not necessarily unique)
{\em canonical representation\/}
\begin{equation}
\label{eq:h:representation}
h=a_1\fprod a_2\fprod \dots\fprod a_m
\
\text{ (for a suitable }
m\in\N^+)
\end{equation}
of every element $h\in V_i$ 
as follows.

{\em Basis of induction\/}.
For $h\in V_n$, we let $h=a_1$ be the canonical representation of $h$.

{\em Inductive step\/}.
Suppose that $i$ is an integer satisfying $0\le i<n$ and we have already defined a canonical representation of every element
$h\in V_{i+1}$.
We fix $h\in V_i$ and define its canonical representation as in 
\eqref{eq:h:representation} according to the rules outlined below.
By \eqref{eq:def:Vi}, at least one (perhaps both) of the following cases holds.

{\sl Case 1\/}. $h\in U_i$. In this case, we let $h=a_1$ be a canonical representation of $h$.

{\sl Case 2\/}. 
$h=x\fprod u\fprod v\fprod x^{-1}$ for 
suitable
  $x\in\dash{X}$ and 
$u,v\in V_{i+1}$. 
Suppose also that 
\begin{equation}
\label{eq:u:v}
u=b_1\fprod b_2\fprod\dots\fprod b_{m_1}
\
\text{ and }
\ 
v=c_1\fprod c_2\fprod\dots\fprod c_{m_2},
\end{equation}
are some canonical representations of $u$ and $v$, respectively.
(These canonical representations were already defined, as $u,v\in V_{i+1}$.)
Then we call 
\begin{equation}
\label{eq:h:gappei}
h=x\fprod 
b_1\fprod b_2\fprod\dots\fprod b_{m_1}
\fprod
c_1\fprod c_2\fprod\dots\fprod c_{m_2}
\fprod x^{-1}
\end{equation}
a canonical representation of $h$.
\end{definition}

\begin{lemma}
\label{reduction:lemma}
Let $\nice{U}=\{U_i:i\le n\}$ 
be a finite sequence satisfying items (a), (b) and (c) of Lemma
\ref{enrichment:lemma}.
Assume that $B'\subseteq B\subseteq F(X)$ and 
$B'$ is symmetric.
Let
$\nice{V}'=\{V_i':i\le n\}$
and
$\nice{V}=\{V_i:i\le n\}$ 
be the $B'$-enrichment and the $B$-enrichment of $\nice{U}$
in $F(X)$, respectively.
Suppose that 
\begin{equation}
\label{eq:12:g}
(B\setminus B')\cap \left(\bar{X}\cup\bigcup_{i=1}^n V_i'\right)\subseteq \{e\}.
\end{equation}

Let $\eta:F(X)\to F(X)$ be the map defined by
\begin{equation}
\label{eq:a_l':new}
\eta(g)=
\begin{cases}
e &\text{if } g\in B\setminus B'\\
g & \text{otherwise}.
 \end{cases}
\end{equation}

If $i\le n$ and $h=a_1\fprod a_2\fprod\dots\fprod a_m$ is a canonical representation  
of some element $h\in V_i$,
then
$
h'=\eta(a_1)\fprod \eta(a_2)\fprod\dots\fprod \eta(a_m)$ is a canonical representation of $h'\in V_i'$.
\end{lemma}
\begin{proof}
We shall prove this lemma by finite reverse induction on $i = n, n-1, \dots, 0$.

First, we shall prove the statement of our lemma for $i=n$.
By Definition \ref{def:enrichment}(i), we have
\begin{equation}
\label{bottom:extension}
V_n'=U_n\cup B'
\text{ and }
V_n=U_n\cup B.
\end{equation}
Fix $h\in V_n$. By Definition 
\ref{canonical:representation:definition},
$h=a_1$ is a canonical representation of $h$.

If $a_1\in B\setminus B'$, then $\eta(a_1)=e\in U_n\subseteq V_n'$ by
\eqref{eq:a_l':new}, (c) and \eqref{bottom:extension}.
Thus, $h'=\eta(a_1)$ is a canonical representation of the element
$h'=e\in V_n'$ by Definition 
\ref{canonical:representation:definition}.

Suppose now that $a_1\not\in B\setminus B'$.
Then $\eta(a_1)=a_1$ by \eqref{eq:a_l':new}, so $h=a_1=\eta(a_1)=h'$.
Since $h\in V_n$, \eqref{bottom:extension} implies that either $h\in U_n$ or $h\in B$. In the former case, $h'=h\in U_n\subseteq V_n'$ by \eqref{bottom:extension}. In the latter case,
from $a_1=h\in B$ and $a_1\not\in B\setminus B'$, we conclude that $a_1\in B'$. Since $B'\subseteq V_n'$ by \eqref{bottom:extension}, it follows that $h'=a_1\in V_n'$.
Thus, $h'=\eta(a_1)$ is a canonical representation of the element
$h'\in V_n'$ by Definition 
\ref{canonical:representation:definition}.

Suppose that that $i<n$ and the statement of our lemma has already been proved for $i+1$. Fix an arbitrary $h\in V_i$.
We consider two cases as in the inductive step of Definition 
\ref{canonical:representation:definition}.

\smallskip
{\em Case 1\/}. $h\in U_i$. In this case $h=a_1$ is a canonical representation of 
$h$ by Case 1 of Definition 
\ref{canonical:representation:definition}.
Since $U_i\subseteq V_i'$, we have $h=a_1\in V_i'$.

If $a_1\in B\setminus B'$, then $a_1=e$ by \eqref{eq:12:g}
and $\eta(a_1)=e$ by \eqref{eq:a_l':new}.
In particular, $h'=\eta(a_1)=e=a_1=h$.
Since $h\in U_i$, we conclude that
$h'=\eta(a_1)$ is a canonical representation of $h'\in V_i'$ by
Case 1 of the inductive step of Definition \ref{canonical:representation:definition}.

Suppose now that $a_1\not\in B\setminus B'$.
Then $\eta(a_1)=a_1$ by \eqref{eq:a_l':new}, so
$h'=\eta(a_1)=a_1=h\in U_i\subseteq V_i'$.
Therefore,
$h'=\eta(a_1)$ is a canonical representation of $h'\in V_i'$ by
Case 1 of the inductive step of Definition \ref{canonical:representation:definition}.

\smallskip
{\em Case 2\/}.  
$h=x\fprod u\fprod v\fprod x^{-1}$ 
for some $x\in \dash{X}$ and
$u,v\in V_{i+1}$. Consider arbitrary 
canonical representations of $u$ and $v$ as 
in \eqref{eq:u:v}, so that \eqref{eq:h:gappei} becomes a canonical representation of $h$.

The following claim holds by our inductive assumption.
\begin{claim}
\label{eq:u:v:prime}
$u'=\eta(b_1)\fprod \eta(b_2)\fprod\dots\fprod \eta(b_{m_1})$
and
$v'=\eta(c_1)\fprod \eta(c_2)\fprod\dots\fprod \eta(c_{m_2})$
are canonical representations of elements $u'\in V_{i+1}'$
and $v'\in V_{i+1}'$, respectively. 
\end{claim}

 \begin{claim}
\label{claim:h*}
$h^*=x\fprod u'\fprod v'\fprod x^{-1}\in V_i'$.
\end{claim}

\begin{proof}
Note that $\nice{V}'$ is a finite neighbourhood system for $F(X)$ by the assumptions of our lemma and Lemma \ref{enrichment:lemma}. In particular,
condition (3$_\nice{V}'$) of Definition \ref{SPSa} holds.
Therefore, 
$$
h^*=x\fprod u'\fprod v'\fprod x^{-1}\in x \fprod V_{i+1}'\fprod V_{i+1}'\fprod  x^{-1}\subseteq V_i',
$$
which implies that $h^*\in V_i'$.
\end{proof}

\begin{claim}
\label{claim:4:s}
$h^*=x\fprod 
\eta(b_1)\fprod \eta(b_2)\fprod\dots\fprod \eta(b_{m_1})
\fprod
\eta(c_1)\fprod \eta(c_2)\fprod\dots\fprod \eta(c_{m_2})
\fprod x^{-1}$ 
is a canonical representation of $h^*\in V_i'$.
\end{claim}
\begin{proof}
This follows from Claims \ref{eq:u:v:prime},  \ref{claim:h*}  and Case 2 of the inductive step of Definition \ref{canonical:representation:definition}.
\end{proof}

\begin{claim}
\label{primes:of:letters} 
$\eta(x)=x$ and $\eta(x^{-1})=x^{-1}$.
\end{claim}
\begin{proof}
  Indeed, if $x\not\in B\setminus B'$, then $\eta(x)=x$ by \eqref{eq:a_l':new}.
If $x \in B\setminus B'$, then $x=e$ by \eqref{eq:12:g},
so $\eta(x)=\eta(e)=e=x$ by \eqref{eq:a_l':new}.
Similarly,
if $x^{-1}\not\in B\setminus B'$, then $\eta(x^{-1})=x^{-1}$
by \eqref{eq:a_l':new}.
If $x^{-1} \in B\setminus B'$, then $x^{-1}=e$ by \eqref{eq:12:g},
so $\eta(x^{-1})=\eta(e)=e=x^{-1}$ by \eqref{eq:a_l':new}.
\end{proof}

Since \eqref{eq:h:gappei} is a canonical representation of $h$ that is now being considered, in order to finish the inductive step, it suffices to show that 
\begin{equation}
\label{eq:h:gappei:prime}
h'=\eta(x)\fprod 
\eta(b_1)\fprod \eta(b_2)\fprod\dots\fprod \eta(b_{m_1})
\fprod
\eta(c_1)\fprod \eta(c_2)\fprod\dots\fprod \eta(c_{m_2})
\fprod \eta(x^{-1})
\end{equation}
is a canonical representation of $h'\in V_i'$.
This follows from
Claims  
\ref{claim:4:s}, \ref{primes:of:letters} and
\eqref{eq:h:gappei:prime}, as $h'=h^*\in V_i'$. 
The inductive step is now complete.
\end{proof}

\begin{lemma}
\label{extension:set:c}
Let $X$ and $Y$ be sets such that $X\not=\emptyset$ and $X\subseteq Y$.  Let $\nice{U}=\{U_i:i\le n\}$ be a \fns\ for $F(X)$ and let $\nice{V}=\{V_i:i\le n\}$ be its 
cyclic $(Y\setminus X)$-enrichment
in $F(Y)$. Then
\begin{equation}
\label{sum:of:letters}
\sum_{l=1}^m |\lett(a_l)|\le |X|\cdot 4^{n-i}
\end{equation}
whenever $i \leq n$
and 
$h=a_1\fprod a_2\fprod\dots\fprod a_m$ is a
canonical representation of $h\in V_i$ as in Definition 
\ref{canonical:representation:definition}.
\end{lemma}
\begin{proof}
We shall prove 
the statement of our lemma
by finite reverse induction on $i = n, n-1, \dots, 0$.

First, we shall prove the statement of our lemma for $i=n$.
By Definition \ref{def:enrichment}, we have
\begin{equation}
\label{eq:def:Vn:Y:W}
V_n=U_n\cup \bigcup_{y\in Y\setminus X} \grp{y}.
\end{equation}
Fix $h\in V_n$. By Definition 
\ref{canonical:representation:definition},
$h=a_1$ is the unique canonical representation of $h$,
so in order to check \eqref{sum:of:letters}, it suffices to show
that $|\lett(h)|\le |X|$.
By \eqref{eq:def:Vn:Y:W}, we need to consider two cases.
If $h \in U_n$, then 
$\lett(h)\subseteq X$, as $U_n\subseteq F(X)$ by (1$_{\nice{U}}$),
so
$|\lett(h)|\le |X|$.
If $h \in \grp{y}$ for some $y \in \dash{Y}$, then $|\lett(h)| \leq 1 \leq |X|$ holds, as $X$ is non-empty. 

Suppose that $i<n$ and the statement of our lemma has already been proved for $i+1$. We shall prove the statement of our lemma for $i$.
Fix $h\in V_i$ and consider an arbitrary canonical representation
$h=a_1\fprod a_1\fprod \dots\fprod a_m$ of $h$.
By Definition 
\ref{canonical:representation:definition},
we need to consider two cases.

If $h\in U_i$, then $m=1$.
Since
$U_i\subseteq F(X)$ by (1$_{\nice{U}})$,
we have $\lett(a_1)=\lett(h)\subseteq X$,
which implies 
$|\lett(a_1)|\le |X|\le|X|\cdot 4^{n-i}$ because $i< n$.

Suppose now that 
there exist $x\in \dash{X}$, $u,v\in V_{i+1}$
and their canonical representations as in \eqref{eq:u:v}
such that
$m=m_1+m_2+2$,
$a_1=x$, $a_m=x^{-1}$,
$a_{1+s}=b_s$ for $1 \le s\le m_1$,
and
$a_{m_1+1+t}=c_t$ for $1 \le t\le m_2$.
Therefore,
\begin{equation}
\label{eq:15:k}
\sum_{l=1}^m |\lett(a_l)|=
|\lett(x)|+
\sum_{l=1}^{m_1} |\lett(b_l)|
+
\sum_{l=1}^{m_2} |\lett(c_l)|
+
|\lett(x^{-1})|.
\end{equation}

By inductive hypothesis, 
\begin{equation} \label{ind:hip:bound}
\sum_{l=1}^{m_1} |\lett(b_l)|\le  |X|\cdot 4^{n-(i+1)}
\text { and } 
\sum_{l=1}^{m_2} |\lett(c_l)|\le  |X|\cdot 4^{n-(i+1)}.
\end{equation}
Since $X$ is non-empty and $i+1\le n$,
we have 
\begin{equation}
\label{eq:17:u}
|\lett(x)|+|\lett(x^{-1})|=2\le 2|X|\cdot 4^{n-(i+1)}.
\end{equation}
Combining \eqref{eq:15:k}, \eqref{ind:hip:bound} and \eqref{eq:17:u},
we conclude that
$$
\sum_{l=1}^m |\lett(a_l)|\le 4 |X|\cdot 4^{n-(i+1)}
=
|X|\cdot 4^{n-i}.
$$ 
This finishes the inductive step.
\end{proof}

\section{Three auxiliary lemmas}
\label{sec:aux}

\begin{lemma}
\label{trivial:sandwich}
Let 
$X$ be a set.
Suppose that $j,n,s\in\N^+$,
$x_1,x_2,\dots,x_s\in X$,
$\varepsilon_1,\varepsilon_2,\dots,\varepsilon_s\in\{-1,1\}$,
\begin{equation}
\label{eq:g:1}
g_0=x_1^{\varepsilon_1}x_2^{\varepsilon_2}\dots x_s^{\varepsilon_s}\in F(X),
\end{equation}
$v_1,\dots,v_n, w_1,\dots,w_{n+1}\in F(X)$ and
\begin{equation}
\label{eq:h}
h^*=w_1\fprod v_1\fprod w_2\fprod v_2\fprod\dots\fprod w_n\fprod v_n\fprod w_{n+1}
\end{equation}
satisfy the following conditions:
\begin{itemize}
\item[(i)]
$x_j\not=x_p$ for $p=1,\dots,s$ with $j\not=p$;
\item[(ii)] 
$x_j\not\in \lett(w_i)$ for all $i=1,\dots,n+1$;
\item[(iii)] $x_j\not\in \lett(h^*)$;
\item[(iv)] for each $i=1,\dots,n$, 
either $v_i=g_0$ or $v_i=g_0^{-1}$.
\end{itemize}
Then $h^* = w^*$, where
\begin{equation}
\label{eq:w}
w^*=w_1\fprod \e{W}\fprod w_2\fprod\dots\fprod w_n\fprod \e{W}\fprod w_{n+1}
\end{equation}
is the word obtained from $h^*$ by replacing all $v_i$ in it with $\e{W}$.
\end{lemma}
\begin{proof}
We use the principle of minimal counter-example.
Suppose that the conclusion of our lemma fails.  Among all counter-examples to our lemma, we choose a counter-example $h^*$ 
for which the number $n$ is the smallest. The goal is to derive a contradiction from this assumption.

It follows from \eqref{eq:g:1}, (i) and (iv) that the variable
$x_j$ appears exactly once in each $v_i$ for $i=1,\dots,n$.
It follows from (ii) that  
the variable $x_j$ does not appear in any of 
$w_1,\dots,w_{n+1}$.
Since
$x_j\not\in \lett(h^*)$ by (iii), all appearances of the terms $x_j$
and $x_j^{-1}$ in \eqref{eq:h} cancel out as a result of computation in $F(X)$ on the right-hand side of \eqref{eq:h}.
Therefore, there exist $i=1,\dots,n-1$ and $\delta\in\{-1,1\}$ such that
the unique term $x_j^\delta$ in $v_i$ cancels in \eqref{eq:h} with the unique 
term $x_j^{-\delta}$ in $v_{i+1}$.  
From \eqref{eq:g:1}, (i) and (iv) one easily concludes that 
the words $v_i$ and $v_{i+1}$ are inverses of each other
and
$v_i\fprod w_{i+1}\fprod v_{i+1}=\e{F}$.
Therefore,
$w_{i+1}=v_i^{-1}\fprod v_{i+1}^{-1}=(v_{i+1}\fprod v_i)^{-1}=\e{F}^{-1}=\e{F}$,
and so
\begin{equation}
\label{eq:w*}
w=w_i\fprod w_{i+2}=w_i
\fprod \e{W}\fprod w_{i+2}=
w_i\fprod v_i\fprod w_{i+1}\fprod v_{i+1}
\fprod
w_{i+2}.
\end{equation}

\begin{claim}
\label{claim:h':m}
The product
\begin{equation}
\label{eq:h'}
h'=w_1\fprod v_1\fprod\dots\fprod w_{i-1}\fprod v_{i-1}\fprod w \fprod v_{i+2}\fprod w_{i+3}\fprod \dots\fprod w_n\fprod v_n\fprod w_{n+1}
\end{equation}
satisfies conditions (i)--(iv) of our lemma, after an obvious re-labeling of its elements.
\end{claim}
\begin{proof}
Indeed, conditions (i) and (iv) are not affected by the change from $h^*$ to $h'$.

Since $\lett(w)\subseteq \lett(w_i)\cup\lett(w_{i+2})$ by the first equation in \eqref{eq:w*} and
Lemma \ref{letters:and:products}(i), 
from item (ii) we conclude that $x_j\not\in \lett(w)$. 
This shows that the representation of $h'$ in \eqref{eq:h'}
satisfies item (ii) of our lemma.

It remains only to check condition (iii).
From 
\eqref{eq:h}, \eqref{eq:w*}
  and
\eqref{eq:h'},
we obtain that 
\begin{equation}
\label{eq:h:h'}
h^*=h'.
\end{equation}
Since $x_j\not\in \lett(h^*)$ by (iii), from \eqref{eq:h:h'} we conclude that $x_j\not\in \lett(h')$.
\end{proof}
  
Since $h'$ 
is ``shorter'' than the minimal counter-example $h^*$ to 
Lemma \ref{trivial:sandwich},
from Claim \ref{claim:h':m} we conclude that
\begin{equation}
\label{eq:w':h'}
h'=w',
\end{equation}
where
\begin{equation}
\label{eq:w'}
w'=w_1\fprod e\fprod \dots\fprod w_{i-1}\fprod e\fprod  w\fprod e\fprod w_{i+3}\fprod \dots\fprod w_n\fprod e\fprod w_{n+1}
\end{equation}
is the word obtained from $h'$ by replacing all $v_i$ in it with $e$.

From \eqref{eq:w}, \eqref{eq:w*}
and \eqref{eq:w'}, we get
$w^*=w'$.
Combining this with
\eqref{eq:h:h'} and
\eqref{eq:w':h'},
we deduce that
$h^*=w^*$.
However, this contradicts the assumption that $h^*$ is a counter-example to our lemma.
\end{proof}

\begin{lemma}
\label{lemma:7.2}
Let 
$X$ be a set.
Suppose that $j,m, s\in\N^+$,
$x_1,x_2,\dots,x_s\in X$,
$\varepsilon_1,\varepsilon_2,\dots,\varepsilon_s\in\{-1,1\}$,
\begin{equation}
\label{eq:g:2}
g_0=x_1^{\varepsilon_1}x_2^{\varepsilon_2}\dots x_s^{\varepsilon_s}\in F(X),
\end{equation}
$a_1,a_2,\dots,a_m\in F(X)$
and
\begin{equation}
\label{eq:def:L}
L=\{l=1,\dots,m: x_j\in\lett(a_l)\}
\end{equation}
satisfy the following conditions:
\begin{itemize}
\item[(a)]
$x_j\not=x_p$ for $p=1,\dots,s$ with $j\not=p$;
\item[(b)] 
$x_j\not\in \lett(a_1\fprod a_2\fprod\dots\fprod a_m)$;
\item[(c)] 
$a_l\in\grp{g_0}\setminus\{e\}$
for each $l\in L$.
\end{itemize}
Let $\eta:F(X)\to F(X)$ be the map which sends each element of the cyclic group $\grp{g_0}$ to $e$ and does not move elements of its complement $F(X)\setminus \grp{g_0}$.
Then 
\begin{equation}
\label{eq:33:z}
a_1\fprod a_2\fprod\dots\fprod a_m
=
\eta(a_1)\fprod \eta(a_2)\fprod\dots\fprod \eta(a_m).
\end{equation}  
\end{lemma}
\begin{proof}
Suppose first that $L=\emptyset$.
Then $x_j\not\in \lett(a_1)\cup \lett(a_2)\cup\dots\cup\lett(a_m)$
by 
\eqref{eq:def:L}.
On the other hand, 
$x_j\in \lett(h)$ for every $h\in\grp{g_0}\setminus \{\e{F}\}$ by Lemma \ref{support:lemma}.
Therefore, $a_l\not\in \grp{g_0}\setminus \{\e{F}\}$ for every 
$l=1,\dots,m$. By our definition of $\eta$, this means that
$\eta(a_l)=a_l$ for all $l=1,\dots,m$.
Thus, \eqref{eq:33:z} holds in this case.
In the rest of the proof we assume that $L\not=\emptyset$.

For every $l\in L$, use item (c) to fix a non-zero integer $q_l$ such that
$a_l=g_0^{q_l}$ and define 
$\delta_l=q_l/|q_l|\in\{-1,1\}$.
It is clear that
\begin{equation}
\label{powers:of:g_0}
a_l=g_0^{\delta_l}\fprod e\fprod g_0^{\delta_l}\fprod e\fprod\dots\fprod e\fprod g_0^{\delta_l},
\end{equation}
where the term $g_0^{\delta_l}$ appears $|q_l|$-many times in \eqref{powers:of:g_0}.
Note that $\eta(e)=e$.
Since $g_0^{q_l},g_0^{\delta_l}\in\grp{g_0}$, we have $\eta(g_0^{q_l})=\eta(g_0^{\delta_l})=e$, and so
\begin{equation}
\label{eta:of:a_l}
\eta(a_l)=\eta(g_0^{q_l})=e=
\eta(g_0^{\delta_l})\fprod \eta(e)\fprod \eta(g_0^{\delta_l})\fprod \eta(e)\fprod\dots\fprod \eta(e)\fprod \eta(g_0^{\delta_l}).
\end{equation}

Let 
$b_1\fprod b_2\fprod \dots\fprod b_t$ be the formal expression obtained from the product $a_1\fprod a_2\dots\fprod a_m$
by replacing in it 
every element $a_l$ 
for $l\in L$ with the formal expression 
on the right-hand side of \eqref{powers:of:g_0}.
Clearly,
\begin{equation}
\label{eq:def:h^*}
a_1\fprod a_2\fprod \dots \fprod a_m
=b_1\fprod b_2\fprod\dots\fprod b_t.
\end{equation}
Since
\eqref{eta:of:a_l}
holds for every $l\in L$,
we obtain
\begin{equation}
\label{eq:39:j}
\eta(a_1)\fprod \eta(a_2)\fprod\dots\fprod \eta(a_m)
=
\eta(b_1)\fprod \eta(b_2)\fprod\dots\fprod \eta(b_t).
\end{equation}

It follows from our definition of $b_r$'s that
\begin{equation}
\label{eq:b_r}
b_r\in\{e, g_0, g_0^{-1}\}\cup \{a_l: l\in \{1,\dots,m\}\setminus L\}
\ 
\text{ for every }
\
r=1,\dots, t.
\end{equation}

Since $L\not=\emptyset$, we can select $l\in L$.
Then $a_l$ has the form as in \eqref{powers:of:g_0}, and so 
$g_0^{\delta_l}$ appears as one of $b_r$'s.
This shows that 
the set
\begin{equation}
\label{eq:def:M}
M=\{r=1,\dots, t: 
\text{either }
b_r=g_0
\text{ or }
b_r=g_0^{-1}\}
\end{equation}
is non-empty, so we can fix an enumeration
$M=\{r_i: i\le n\}$
  of $M$ such that 
$r_1<r_2<\dots<r_n$. For $i=1,\dots,n$, define
$
v_i=b_{r_i}.
$

If $r_1=1$, we let $w_1=e$; otherwise, we let
$w_1=b_1\fprod \dots \fprod b_{r_1-1}$. Similarly, if 
$r_n=t$, we let $w_{n+1}=e$; otherwise, we let
$w_{n+1}=b_{r_n+1}\fprod\dots\fprod b_t$.
For $i=2,\dots,n$, we let
$w_i=b_{r_{i-1}+1}\fprod \dots\fprod b_{r_i-1}$.
It follows from 
our definition of $v_i$ and $w_i$ that
\begin{equation}
\label{eq:40:u}
b_1\fprod b_2\dots\fprod b_t=
w_1\fprod v_1\fprod w_2\fprod v_2\fprod\dots\fprod w_n\fprod v_n\fprod w_{n+1}.
\end{equation}

Let $h^*$ be the element defined in \eqref{eq:h}. Note that
\begin{equation}
\label{eq:41:g}
a_1\fprod a_2\fprod \dots\fprod a_m=h^*
\end{equation}
by
\eqref{eq:h}, \eqref{eq:def:h^*} and \eqref{eq:40:u}.

\begin{claim}
$h^*$
satisfies conditions (i)--(iv) of
Lemma \ref{trivial:sandwich}. 
\end{claim}
\begin{proof}
Condition (i) of Lemma \ref{trivial:sandwich} coincides with condition (a) of our lemma.

(ii)
Let $i=1,\dots,n+1$ be arbitrary.
Combining our definition of $w_i$ with 
\eqref{eq:b_r} and \eqref{eq:def:M}, 
we conclude that $w_i$ is a finite product of elements
of the set 
$\{e\}\cup \{a_l: l\in \{1,\dots,m\}\setminus L\}$.
Since $\lett(e)=\emptyset$ and $x_j\not\in \lett(a_l)$ for $l\in \{1,\dots,m\}\setminus L$ by \eqref{eq:def:L},
applying Lemma \ref{letters:and:products}(ii) we obtain that $x_j\not\in\lett(w_i)$.

(iii) 
Note that 
$x_j\not\in \lett(h^*)$ by \eqref{eq:41:g} and (b).

(iv)
If
$i=1,\dots,n$, then $v_i=b_{r_i}\in M$ by the choice of our enumeration of $M$, so \eqref{eq:def:M} implies that
either $v_i=g_0$ or $v_i=g_0^{-1}$. 
\end{proof}

Applying Lemma \ref{trivial:sandwich}, we conclude that
\begin{equation}
\label{eq:42:k}
h^*=w_1\fprod \e{W}\fprod w_2\fprod\dots\fprod w_n\fprod \e{W}\fprod w_{n+1}.
\end{equation}

Note that $\eta(b_r)=e$ for every $r\in M$.
Suppose now that $r\in\{1,\dots,t\}\setminus M$.
Since $x_j\in \lett(g)$ for every $g\in\grp{g_0}\setminus \{e\}$
and $x_j\not\in \lett(b_r)$, it follows that $b_r\not\in\grp{g_0}\setminus \{e\}$. Therefore,
$\eta(b_r)=b_r$.
This argument and our definition of $w_r$'s implies that
\begin{equation}
\label{eq:43:j}
w_1\fprod \e{W}\fprod w_2\fprod\dots\fprod w_n\fprod \e{W}\fprod w_{n+1}
=
\eta(b_1)\fprod \eta(b_2)\fprod\dots\fprod \eta(b_t).
\end{equation}
Now \eqref{eq:33:z} follows from 
\eqref{eq:41:g}, 
\eqref{eq:42:k},
\eqref{eq:43:j} 
and \eqref{eq:39:j} (in this order). 
\end{proof}

\begin{lemma} \label{main:lemma:assgp}
Assume that  $X$ is a non-empty finite set,  $g\in F(X)$,
$\nice{U} = \{U_i : i \leq n\}$ is a \fns\ for $F(X)$ and $Y$ is a finite set containing $X$ satisfying the inequality
$|Y\setminus X|>|X|\cdot 4^{n}$.
Then
there exists 
a  
\fns\ $\nice{V} = \{V_i : i \leq n\}$ 
for $F(Y)$
extending $\nice{U}$ such that $g \in  \grp{\Cyc(V_n)}$.
\end{lemma}

\begin{proof}
Let 
$|Y\setminus X|=k$. By the assumption of our lemma, we have
\begin{equation}
\label{eq:k}
k>  |X|\cdot 4^{n}.
\end{equation}
Fix faithful enumeration
$Y\setminus X = \{y_1, \dots, y_k\}$  of 
the set $Y\setminus X$
and
define
\begin{equation} 
\label{g0:def} 
g_0=(y_1 y_2\dots y_k)\fprod g.
\end{equation}
Note that the product in \eqref{g0:def} does not undergo any cancellations, as $\lett(g)\subseteq X$ and 
the set $\{y_1, \dots, y_k\}=Y\setminus X$ is disjoint from $X$.
For the same reason, for every non-zero integer $p$, 
the power 
$g_0^p$ of $g_0$ does not undergo any cancellations as well.
Therefore,
\begin{equation}
\label{support:of:g0}
\{y_1,y_2,\dots,y_k\}\subseteq \lett(g_0^p)
\
\text{ for every non-zero integer }
p.
\end{equation}

Since $X\subseteq Y$, we have $F(X)\subseteq F(Y)$. Therefore, 
the finite sequence $\nice{U}$ satisfies conditions (a)--(c) of Lemma \ref{enrichment:lemma}, where in condition (a) one has to replace $X$ by $Y$.
Define 
\begin{equation}
\label{eq:34:j}
B'=\bigcup_{y\in Y\setminus X} \grp{y}
\
\text{ and }
\ 
B=B'\cup \grp{g_0}. 
\end{equation}
Clearly, $B'\subseteq B\subseteq F(Y)$ and 
$B'$ is symmetric.

\begin{claim}
\label{extension:claim:1}
The $B'$-enrichment
$\nice{V}'=\{V_i':i\le n\}$ of $\nice{U}$ in $F(Y)$ is a \fns\ for $F(Y)$ extending $\nice{U}$.
\end{claim}
\begin{proof}
From the first equation in \eqref{eq:34:j}, 
Definition \ref{def:enrichment}(ii) and definition of $\nice{V}'$, it follows that $\nice{V}'$ coincides with the cyclic $(Y\setminus X)$-enrichment $\nice{U}_Y$ of 
$\nice{U}$.
Now the conclusion of our claim follows from
Corollaries \ref{cor:2} and \ref{extension:set:b}.
\end{proof}

\begin{claim}
\label{claim:8:k}
The $B$-enrichment
$\nice{V}=\{V_i:i\le n\}$ of $\nice{U}$ in $F(Y)$ is a \fns\ for $F(Y)$.
\end{claim}
\begin{proof}
It follows from \eqref{eq:34:j}
that $B^{-1}=B$. Now the conclusion of our claim follows from Lemma \ref{enrichment:lemma} (in which one has to replace $X$ by $Y$).
\end{proof}

\begin{claim} \label{first:claim} 
$g \in \grp{\Cyc(V_n)}$.
\end{claim}  
\begin{proof}
Since $Y\setminus X = \{y_1, \dots, y_k\}$, 
from \eqref{eq:34:j} we conclude that
$\grp{y_i^{-1}}= \grp{y_i} \subseteq B$
for all $i = 1, \dots, k$. 
Similarly,
$\grp{g_0}\subseteq B$ by \eqref{eq:34:j}.
Since 
$\nice{V}$ is the $B$-enrichment of $\nice{U}$,
from 
equation \eqref{eq:def:Vn} of Definition \ref{def:enrichment},
we conclude that
$B\subseteq V_n$.
This implies that $y_1^{-1}, \dots, y_k^{-1}, g_0 \in \Cyc(V_n)$.
From this and \eqref{g0:def}, we get
$g=y_k^{-1}\fprod y_{k-1}^{-1}\fprod\dots\fprod y_1^{-1}\fprod g_0\in \grp{\Cyc(V_n)}$.
\end{proof}

\begin{claim}
The map $\eta:F(X)\to F(X)$ defined in Lemma \ref{lemma:7.2}
coincides with the map $\eta$ defined in 
Lemma \ref{reduction:lemma}.
\end{claim}
\begin{proof}
From \eqref{eq:34:j} we get $B\setminus B'=\grp{g_0}\setminus \{e\}$. The conclusion of our claim follows from this observation and our definitions of both maps.
\end{proof}

In the rest of the proof we shall denote both of the maps from the above claim by $\eta$.

\begin{claim}
\label{claim:11:d}
$(B\setminus B')\cap \left(\bar{Y}\cup\bigcup_{i=1}^n V_i'\right)=\emptyset$.
\end{claim}
\begin{proof}
Suppose that $h\in (B\setminus B')\cap \left(\bar{Y}\cup\bigcup_{i=1}^n V_i'\right)$. 
From $h\in B\setminus B'$ and \eqref{eq:34:j}, we conclude that
$h\in \grp{g_0}\setminus\{0\}$, 
so $h=g_0^p$ for some non-zero integer $p$.
Then
$|\lett(h)|\ge |Y\setminus X|=k$ by \eqref{support:of:g0}.
Note that $h\not\in\dash{Y}$, as $|\lett(y)|\leq 1<k$ for every $y\in \dash{Y}$.  
Since   $h\in \bar{Y}\cup\bigcup_{i=1}^n V_i'$, it follows that 
$h\in V_i'$ for some $i\le n$.
Since $\nice{V}'$ is the $B'$-enrichment of $\nice{U}$,
the element $h\in V_i'$ has a canonical representation
\begin{equation}
\label{h:representation}
h=a_1\fprod a_2\fprod \dots\fprod a_m.
\end{equation}
Then 
$\lett(h)\subseteq \bigcup_{l=1}^m \lett(a_l)$ by 
\eqref{h:representation} and Lemma \ref{letters:and:products}(ii).
Furthermore,
$$
|\lett(h)|\le \left|\bigcup_{l=1}^m \lett(a_l)\right|\le\sum_{l=1}^m |\lett(a_l)|\le |X|\cdot 4^{n-i}\le |X|\cdot 4^n<k
$$
by Lemma \ref{extension:set:c} and \eqref{eq:k},
in contradiction with $|\lett(h)|\ge k$.
\end{proof}

\begin{claim} \label{claim:12}
Suppose $h\in V_i\cap F(X)$ has a canonical representation 
 \eqref{h:representation}. Then \begin{equation} \label{eq:48:h}
 h' = 
 \eta(a_1) \cdot \eta(a_2) \cdot \cdots \cdot \eta(a_m)
 \end{equation}  is a canonical representation of 
 $h' \in V_i'$
satisfying 
 the inequality
 $\left|\bigcup_{l=1}^m \lett(\eta(a_l))\right| <k$.
\end{claim}

\begin{proof}
Let $h\in V_i\cap F(X)$ be arbitrary.
From Claim \ref{claim:11:d}, we conclude that \eqref{eq:12:g} holds (with $X$ replaced by $Y$). Therefore, all the assumptions of Lemma 
\ref{reduction:lemma} are satisfied (with $X$ replaced by $Y$ in this lemma).
This implies that $h'= \eta(a_1) \cdot \eta(a_2) \cdot \cdots \cdot \eta(a_m)$ is a canonical representation of $h' \in V_i'$. 
Finally, 
$$
\left|\bigcup_{l=1}^m \lett(\eta(a_l))\right| \le |X|\cdot 4^{n-i}\le |X|\cdot 4^n<k
$$ 
by 
Lemma \ref{extension:set:c} and \eqref{eq:k}. 
\end{proof}

\begin{claim} \label{claim:13}
If $h\in V_i\cap F(X)$ has a canonical representation 
 \eqref{h:representation}
and  $a_l \neq \eta(a_l)$ for
some $l  = 1, \dots, m$,
then $a_l \in \grp{g_0} \setminus \{e\}$. 
\end{claim}

\begin{proof}
Suppose that $a_l \neq \eta(a_l)$
for some $l  = 1, \dots, m$.
By \eqref{eq:a_l':new}, this implies that $a_l \neq e$ and furthermore $a_l \in B \setminus B'$. Finally, $B \setminus B' \subseteq \grp{g_0}$ by  \eqref{eq:34:j}. This establishes that $a_l \in \grp{g_0} \setminus \{e\}$. 
\end{proof}

\begin{claim}
\label{assumptions:satisfied}
If $h\in V_i\cap F(X)$ has a canonical representation 
 \eqref{h:representation},
then
elements $g_0$, $a_1,\dots,a_m\in F(X)$ 
satisfy all assumptions of
Lemma \ref{lemma:7.2}.
\end{claim}
\begin{proof}

Since $\left|\bigcup_{l=1}^m \lett(\eta(a_l))\right| <k$ by
 Claim \ref{claim:12},
we can use 
\eqref{g0:def} to choose some variable $y_j \in \{y_1, \dots, y_k \}$  such that $y_j \not \in \lett(\eta(a_l))$ for all $l = 1, \dots m$.
Let
\begin{equation} 
\label{33:satisfied}
L = \{l = 1, \dots, m: y_j \in \lett(a_l)\}
\end{equation}
and
\begin{equation}
M=\{ l = 1, \dots,m : a_l \neq \eta(a_l) \}.
\end{equation}

By \eqref{support:of:g0}, we have $y_j \in \lett(w)$
for every $w \in \grp{g_0} \setminus \{e\}$. This implies that $M \subseteq L$ by Claim \ref{claim:13}.
To show the reverse inclusion, let us suppose that 
$l\in L$. Then
$y_j \in \lett(a_l)$ by \eqref{33:satisfied}.
Since $y_j \not \in \lett(\eta(a_l))$ by our choice of $a_l$, this implies that $a_l \neq \eta(a_l)$; that is, $l\in M$. 
This shows that $L\subseteq M$.
From these two inclusions, we get 
$L=M$.

We are now ready to check the assumptions (a)--(c) of Lemma \ref{lemma:7.2} with $y_j$ taken as $x_j$.

(a) 
Suppose that \eqref{eq:g:2} holds.
First, observe that $y_j \not \in lett(g_0)$ since $g_0 \in F(X)$ and $y_j \in Y \setminus X$. Next, for every $i = 1, \dots, k$ we have that $y_j \neq y_i$ if $i \neq j$ by our faithful enumeration of $Y \setminus X$. Recalling
\eqref{g0:def}, we conclude that (a) holds.

(b) Recall that $h \in V_i \cap F(X)$ by hypothesis, so 
$\lett(h)\subseteq X$.
Since $y_j \in Y \setminus X$, we have 
$y_j \not \in \lett(h)$. It remains only to note that $h=a_1\fprod a_2\fprod\dots\fprod a_m$.

(c) Given $a_l$ with $l \in L$, we have $a_l \neq \eta(a_l)$, as $L=M$.
Now $a_l \in \grp{g_0} \setminus \{e\}$
by Claim \ref{claim:13}.
\end{proof}

\begin{claim}
\label{Vi:claim}
$V_i\cap F(X)\subseteq U_i$ for every $i\le n$.
\end{claim}
\begin{proof}
Let $h\in V_i\cap F(X)$ be an arbitrary element, and let \eqref{h:representation} be one of its canonical representations. Claim \ref{assumptions:satisfied} allows us to make use of Lemma \ref{lemma:7.2} to show that $h = \eta(a_1)\fprod \eta(a_2) \fprod\dots\fprod \eta(a_m).$ From Claim \ref{claim:12} and \eqref{eq:48:h}, we can conclude that $h = h'$, where $h' = \eta(a_1)\fprod \eta(a_2) \fprod\dots\fprod \eta(a_m)$ is the canonical representation of $h' \in V_i'$.  
In particular, the equality $h = h'$ shows that $h=h'\in V_i'\cap F(X)$.
Since 
$\nice{V}'$ extends $\nice{U}$ by Claim
\ref{extension:claim:1},
$V_i'\cap F(X)=U_i$ by (iii$_{\nice{V}'}^{\nice{U}}$).
This shows that $h\in U_i$.
\end{proof}

\begin{claim} \label{extension:claim}
The \fns\ $\nice{V}$ for $F(Y)$ extends $\nice{U}$.
\end{claim}
\begin{proof}
Since $n \leq n$ and $X \subseteq Y$,  conditions(i$_{\nice{V}}^{\nice{U}} $) and (ii$_{\nice{V}}^{\nice{U}} $) of Definition \ref{SPSb} are both satisfied. 

Let $i\le n$ be arbitrary.
Since $\nice{V}$ is the $B$-enrichment of $\nice{U}$,
$U_i\subseteq V_i$ by equation \eqref{eq:def:Vi} of Definition \ref{def:enrichment}. Since $\nice{U}$ is a \fns\ for $F(X)$, we have
$U_i\subseteq F(X)$ by (1$_{\nice{U}}$).
This shows that 
$U_i\subseteq V_i\cap F(X)$.
The inverse inclusion holds by Claim \ref{Vi:claim}.
Thus,
$V_i\cap F(X)=U_i$.
Since this equation holds for an arbitrary $i\le n$,
condition
(iii$_{\nice{V}}^{\nice{U}} $) also holds.
\end{proof}
The conclusion of our lemma follows 
from Claims \ref{claim:8:k}, \ref{first:claim} and \ref{extension:claim}.
\end{proof}

\section{The partially ordered set and density lemmas}
\label{Sec:6}

\begin{definition}
\label{def:P}
Let $X$ be an infinite set.
\begin{itemize}
\item[(a)] 
We denote by 
$\P$ the set of all triples $p=\ll X^p,n^p,\nice{U}^p\gg$ satisfying the following conditions:
\begin{itemize}
\item[(1$_p$)] $X^p\in[X]^{<\omega}$,
\item[(2$_p$)] $n^p\in\N$,
\item[(3$_p$)] $\nice{U}^p = \{U_i^p:i \leq n^p\}$ is a finite neighbourhood system for $F(X^p)$.
\end{itemize}

\item[(b)] Given triples
$
p=\ll X^p,n^p,\nice{U}^p \gg\in\P
$
and 
$q=\ll X^q,n^q,\nice{U}^q,\gg\in\P,
$ we define
$q\le p$ if and only if 
$\nice{U}^q$ is an extension of $\nice{U}^p$ in the sense of Definition \ref{SPSb}.
\end{itemize}
\end{definition}

The following lemma easily  follows from item (b) of
Definition \ref{def:P}.

\begin{lemma}
The pair $(\P,\le)$ 
is a partially ordered set.
\end{lemma}

\begin{lemma}
There exists $p\in \P$ such that $X^p\not=\emptyset$. In particular, 
$\P\not=\emptyset$.
\end{lemma}
\begin{proof}
Let $X^p = \{x\}$ be a non-empty singleton set. If we let $n^p = 0$ and $U_0^p = \{e\}$, then
$$p=\ll X^p,n^p,\{U_i^p:i \leq 0\}\gg
=
\ll \{x\},0,\{U_0^p\}\gg$$
clearly satisfies conditions (1$_p$)--(3$_p$).
\end{proof}

\begin{definition}
Let $(\P,\le)$ be a poset.
Recall that a set $D\subseteq \P$ is called:
\begin{itemize}
\item[(i)] 
   {\em dense in $(\P,\le)$\/} provided that for every $p\in \P$ there exists $q\in D$ such that $q\le p$;
\item[(ii)]   
   {\em downward-closed in $(\P,\le)$\/} if $p \in D$, $q \in \P$ and $q \leq p$ imply $q \in D$.
\end{itemize}    
\end{definition}

The relation between these two notions is made apparent by the following straightforward lemma.

\begin{lemma}
\label{downwardclosed:dense:sets}
If $A,B \subseteq \P$ are dense subsets of a poset $(\P,\le)$ and $A$ is downward-closed, then $A \cap B$ is dense in $(\P,\le)$.
\end{lemma}

\begin{lemma} \label{dense:sets}
\begin{itemize}
\item[(i)]For every $n\in\N$, the set $A_n=\{q\in\P: n\le n^q\}$ is dense and downward-closed in $(\P,\le)$.
\item[(ii)]For every $S \in [X]^{< \omega}$, the set $B_S = \{ q \in \P : S \subseteq X^q \}$ is dense and downward-closed in $(\P,\le)$.
\item[(iii)]For every 
$g \in F(X)\setminus\{e\}$, the set $C_g = \{q \in \P: g\in F(X^q)\setminus U^q_{n^q}\}$ is dense in $(\P,\le)$. 
\item[(iv)]For every word $g \in F(X)$ the set $D_g = \{q \in \P: g \in \grp{ \Cyc (U_{n^q}^q ) } \}$ is dense in $(\P,\le)$.
\end{itemize}

\end{lemma}

\begin{proof}
(i) Fix $n \in \N$. To prove that $A_n$ is dense in $(\P,\le)$, we consider an arbitrary $p \in \P$. If $n \leq n^p$, the inclusion $p \in A_n$ holds trivially. For this reason, we can assume that $n^p < n$.  By Lemma \ref{trivial-extension:by:e}, there exists a finite neighbourhood system $\nice{V} = \{V_i: i \leq n\}$ for $F(X^p)$ extending $\nice{U}^p$. If we define $X^q = X^p$, $n^q = n$ and $\nice{U}^q = \nice{V}$,
 then 
$q = \ll X^q,n^q,\nice{U}^q\gg \in \P.$
Since $\nice{U}^q$ extends $\nice{U}^p$, 
we have $q \leq p$
by Definition \ref{def:P}(b).
Finally, 
$n = n^q$ implies $q \in A_n$ by the definition of $A_n$, thereby showing that $A_n$ is dense in $(\P,\le)$. 

To check that $A_n$ is downward-closed, consider $p \in A_n$ and $q \in \P$ which satisfy $q \leq p$. 
Since $\nice{U}^q$ extends $\nice{U}^p$ by Definition \ref{def:P}(b), the inequality $n^p \leq n^q$ holds by item
(ii$^{\nice{U}^p}_{\nice{U}^q}$) of Definition 
\ref{SPSb}.
Since $p \in A_n$, we also have $n \leq n^p$. By transitivity, this implies that $n \leq n^q$, and therefore $q \in A_n$ by definition of $A_n$. This shows that $A_n$ is downward-closed.

(ii) Let $S \in [X]^{< \omega}$ and $p \in \P$ be arbitrary. 
Then $X^p\in [X]^{<\omega}$ by (1$_p$),
so
$X^q=X^p\cup S\in [X]^{<\omega}$ holds as well.
Let $n^q = n^p$.
Since $\nice{U}^p$ is a finite neighbourhood system for $F(X^p)$
by (3$_p$), the cyclic $(S \setminus X^p)$-enrichment $\nice{U}^q$ of $\nice{U}^p$ is a \fns\ for $F(X^q)$
by Corollary \ref{cor:2}.
We have shown that
$q = \ll X^q,n^q,\nice{U}^q\gg\in\P$.
Since $\nice{U}^q$ is an extension for $\nice{U}^p$ by Corollary \ref{extension:set:b}, 
we have $q\le p$ by Definition \ref{def:P}(b).
Finally $S\subseteq X^q$ by definition of $X^q$,
so $q\in B_S$ by definition of $B_S$.
This shows that $B_S$ is dense in $(\P,\le)$.

To prove that $B_S$ is downward-closed, 
let $p \in B_S$, $q \in \P$ and $q \leq p$. 
Then $S\subseteq X^p$ by definition of $B_S$.
By condition (i$_q^p$), we have that $X^p \subseteq X^q$, which implies that $S \subseteq X^p \subseteq X^q$. Therefore, $q \in B_S$ by definition, and so $B_S$ is downward-closed.

(iii) Let $g \in F(X)\setminus\{e\}$ 
and $p \in \P$ be arbitrary. Since $g$ is generated by the symbols in $X$, there exists some finite set $X_g \subseteq X$ such that $g \in \grp{X_g}$. Since the set $B_{X_g}$ is dense in $\P$ by (ii), 
we may assume  without loss of generality that $p \in B_{X_g}$. Since $p \in B_{X_g}$, this implies that 
$X_g\subseteq X^p$, so
$g \in F(X^p)$. By Lemma \ref{trivial-extension:by:e}, there exists a finite neighbourhood system $\nice{U}^q = \{U^q_i : i \leq n^p+1\}$ on $X^p$ which extends $\nice{U}^p$ such that $g \not \in U^q_{n^p+1} = \{e\}$. If we define $n^q = n^p+1$ and $X^q = X^p$, then $q = \ll X^q,n^q,\nice{U}^q\gg \in \P$. 
Since $\nice{U}^q$ extends $\nice{U}^p$, we have $q\le p$
by Definition \ref{def:P}(b).

(iv) 
Let $g\in F(X)$ and $p\in \P$. 
Arguing as in the proof (iii), we may find a finite subset $X_g$ of $X$ such that $g\in F(X_g)$. 
By (ii), without loss of generality, we may assume that $p\in B_{X_g}$.
Since the set $X^p \subseteq X$ is finite, we can fix a set $X^q \subseteq X$ containing $X^p$  such that $|X^q\setminus X^p| = k > |X^p| \cdot 4^{n^p}$. 
By Lemma \ref{main:lemma:assgp}, there exists 
a finite neighbourhood system $\nice{U}^q = \{U^q_i: i \leq n^p\}$ on $X^q$ extending $\nice{U}^p$ such that $g \in \grp{\Cyc (U_n^q) }$. If we define $n^q = n^p$, then we have that $q = \ll X^q,n^q,\nice{U}^q\gg \in \P$. 
Since $\nice{U}^q$ extends $\nice{U}^p$, we have $q\le p$
by Definition \ref{def:P}(b).
Since the condition $g \in \grp{\Cyc (U_n^q) }$ is satisfied, we have $q \in D_g$. 
\end{proof}

\section{Proof of Theorem \ref{countable:theorem}}
\label{Sec:7}

We shall need the following folklore lemma an easy proof of which can be found in either \cite[Lemma 9.1]{Countable_SSGP} or \cite[Lemma 14]{Axioms}.

\begin{lemma}
\label{countable:generic:filter}
If $\mathscr{D}$ is an at most countable family of dense subsets of a non-empty poset 
$(\P,\le)$, then there exists an at most countable subset $\F$ of $\P$
such that $(\F,\le)$ is a linearly ordered set and 
$\F\cap D\not=\emptyset$ for every $D\in\mathscr{D}$.
\end{lemma}

Let $X$ be a countably infinite set. 
Since $X$ algebraically generates the free group $F(X)$, the latter is at most countable.

Let $(\P,\le)$ be the poset from Definition \ref{def:P}
which uses the set $X$ as its parameter. 
Clearly,
the family 
\begin{equation}
\label{family:D}
\mathscr{D}=
\{C_g: g\in F(X) \setminus\{e\}\}
\cup
\{A_n \cap D_g :n \in \N, g \in F(X)\}
\cup
\{B_F: F \in [X]^{< \omega} \}
\end{equation}
of subsets of $\P$ is at most countable. 

Let us check that
all members of $\mathscr{D}$ are dense in $(\P,\le)$. 
By Lemma \ref{dense:sets}(iii),
each $C_g$ for $g\in F(X) \setminus\{e\}$ is dense in 
$(\P,\le)$.
Let $n \in \N$ and $g \in F(X)$ be arbitrary.
Since $A_n$ is dense and downward-closed in $(\P,\le)$
by Lemma \ref{dense:sets}(i),
and 
$D_g$ is dense in $(\P,\le)$ by Lemma \ref{dense:sets}(iv), 
by Lemma \ref{downwardclosed:dense:sets} we can conclude that
$A_n\cap D_g$ is dense in $(\P,\le)$. 
Finally, the density in $(\P,\le)$ of each $B_F$ for $F\in[X]^{<\omega}$ follows from Lemma \ref{dense:sets}(ii).

Since we have shown that all members of $\mathscr{D}$ are dense in $\P$, we can apply Lemma \ref{countable:generic:filter} to find a countable
set $\F\subseteq \P$ such that $(\F,\le)$ is a linearly ordered set and 
$\F\cap D\not=\emptyset$ for every $D\in\mathscr{D}$.
For every $n\in\N$,  define
\begin{equation}
\label{def:U_i}
U_n=\bigcup\{U_n^p:p\in\F\text{ and }n\le n^p\}.
\end{equation}

Our nearest goal is to show that the family 
\begin{equation}
\label{eq:fam:B}
\B = \{U_n:n\in\N\}
\end{equation}
is a neighbourhood base at $\e{F}$ 
of a Hausdorff group topology $\mathscr{T}$ on the free group $F(X)$.
The verification of this will be split into a sequence of claims.

\begin{claim}
\label{claim:6}
The equality $\bigcap_{n\in\N} U_n = \{\e{F}\}$ holds.
\end{claim}

\begin{proof}
Let us show that $e \in \bigcap_{n\in\N} U_n$. Take an arbitrary $n \in \N$. Since $A_n  \in \mathscr{D}$ by \eqref{family:D}, the choice of $\F \subseteq \P$ allows us to find some $p \in A_n \cap \F$. 
Then $n\le n^p$ by the definition of $A_n$.
Since $p\in\P$, the  family $\nice{U}^p = \{U_i^p:i \leq n^p\}$ is a finite neighbourhood system for $F(X^p)$ by condition (3$_p$) of Definition \ref{def:P}(a). Therefore, 
we can apply Remark \ref{identity:remark} to conclude that $e\in U_n^p$.
Since $p\in \F$ and $n\le n^p$, we have $U_n^p\subseteq U_n$ by \eqref{def:U_i}.
Since $n \in \N$ was arbitrary, we conclude that $e \in \bigcap_{n \in \N} U_n$.

Suppose that there exists some $g \in \bigcap_{n \in \N} U_n$ with $g\not=e$. Then $C_g \in \mathscr{D}$ by \eqref{family:D}. The choice of $\F$ allows us to find some $p \in C_g \cap \F$. By the definition of $C_g$, this automatically implies that \begin{equation} \label{g:is:out}
g \in F(X^p) \setminus U_{n^p}^p.
\end{equation}
Since $g \in \bigcap_{n \in \N} U_n$, the inclusion $g \in U_{n^p}$ holds. 
By \eqref{def:U_i},
this implies the existence of $q \in \F$ such that $n\le n^q$ and $g \in U_{n^p}^q$. Since $\F$ is linearly ordered, either $q \leq p$ or $p \leq q$. We shall show that both of these two conditions lead to a contradiction.  

If $q \leq p$ holds, then $\nice{U}^q$ is an extension of $\nice{U}^p$ by Definition \ref{def:P}(b), so conditions
(i$_{\nice{U}^q}^{\nice{U}^p}$) and 
(iii$_{\nice{U}^q}^{\nice{U}^p}$) of Definition \ref{SPSb} imply that
$F(X^p)\subseteq F(X^q)$ 
and 
$U_{n^p}^q \cap F(X^p) = U_{n^p}^p$.
Since 
$g$ belongs to the set on the left-hand side of the last equation,
we get
$g \in U_{n^p}^p$. This is a direct contradiction to \eqref{g:is:out}, so this case cannot hold. 

Suppose that $p \leq q$ holds.  Then this time $\nice{U}^p$ is an extension of $\nice{U}^q$ by Definition \ref{def:P}(b),
so
 conditions
(i$_{\nice{U}^p}^{\nice{U}^q}$) and  
   (iii$_{\nice{U}^p}^{\nice{U}^q}$) of Definition \ref{SPSb} imply
    that 
$F(X^q)\subseteq F(X^p)$ 
and         
$
U_{n^p}^p \cap F(X^q) = U_{n^p}^q.
$
Since $g \in U_{n^p}^q$, the last equation implies that $g \in U_{n^p}^p$. Once more, this contradicts \eqref{g:is:out}, so this case cannot hold either. 

The obtained contradiction finishes the proof of our claim.
\end{proof}

\begin{claim}
\label{claim:7}
$U_{n}^{-1} = U_{n}$ and $U_{n+1} \fprod U_{n+1} \subseteq U_{n}$ for every $n \in \N$. 
\end{claim}
\begin{proof}
Fix $n \in \N$. 

Consider an arbitrary $p \in \F$ satisfying $n \leq n^p$. Then
$U^p_n = (U^p_n)^{-1}$ by
  condition (2$_\nice{U}$) of Definition \ref{SPSa}. If we apply this to \eqref{def:U_i}, we obtain that $U_n = U_n^{-1}$.

For the second inclusion, consider some arbitrary $g_1,g_2 \in U_{n+1}$. By  \eqref{def:U_i}, there exist $p_1,p_2 \in \F$ such that $n + 1 \leq n^{p_j}$ and $g_j \in U_{n+1}^{p_j}$ for $j = 1,2$.  Since $\F$ is linearly ordered, without loss of generality, we may assume that $p_1 \leq p_2$. 
Then $\nice{U}^{p_1}$ extends $\nice{U}^{p_2}$
by Definition \ref{def:P}(b).
Now conditions (i$_{\nice{U}^{p_1}}^{\nice{U}^{p_2}}$)  and (iii$_{\nice{U}^{p_1}}^{\nice{U}^{p_2}}$) of Definition \ref{SPSb} imply
$F(X^{p_2})\subseteq F(X^{p_1})$ and
$
U_{n+1}^{p_1} \cap F(X^{p_2}) = U_{n+1}^{p_2}.
$
From the last equation, we obtain $g_1,g_2 \in U_{n+1}^{p_1}$. 
Note that $\nice{U}^{p_1}$ is a \fns\ for $F(X^{p_1})$ by 
condition (3$_{p_1}$) of 
Definition \ref{def:P}(a). 

Since $e \in \overline{X^{p_1}}$, then by using condition (3$_{ \nice{U}^{p_1}}$) of Definition \ref{SPSa} we obtain that \begin{equation}
\label{eq:56:h}
g_1\fprod g_2\in U^{p_1}_{n+1} \cdot U^{p_1}_{n+1} =
e\fprod U^{p_1}_{n+1} \cdot U^{p_1}_{n+1}\fprod e^{-1}
\subseteq 
\bigcup_{x\in\dash{X}} x\fprod U^{p_1}_{n+1} \cdot U^{p_1}_{n+1}\fprod x^{-1}\subseteq
U^{p_1}_{n}.
\end{equation}
Since $p_1 \in \F$ and $n <n+1\leq n^{p_1}$, we have 
$U^{p_1}_{n}\subseteq U_n$ by \eqref{def:U_i}.
From this inclusion and \eqref{eq:56:h},  
we obtain $g_1 \cdot g_2 \in U_n$. 
Since $g_1, g_2 \in U_{n+1}$ were arbitrary, this establishes 
the inclusion $U_{n+1} \cdot U_{n+1} \subseteq U_n$.
\end{proof}
\begin{claim}
\label{untangle}
For every $n \in \N$ and each $y \in \dash{X}$, the inclusion $y \fprod U_{n+1} \fprod y^{-1} \subseteq U_n$ holds.
\end{claim}

\begin{proof}
Let $n \in \N$ and $y \in \dash{X}$ be arbitrary. 
It follows from Definition \ref{def:bar} that
$y=x^\varepsilon$ for some $x\in X$ and $\varepsilon\in\{-1,1\}$.
Since $B_{\{x\}} \in \mathscr{D}$, there exists $p \in B_{\{x\}} \cap \F$ by the choice of $\F$. Let $h \in U_{n+1}$. By \eqref{def:U_i}, there exists some $q \in \F$ such that $n +1 \leq n^q$ and  $h \in U_{n+1}^q$. Since $\F$ is linearly ordered, there exists some $r \in \F$ such that $r \leq p$ and $r \leq q$. 
Then $\nice{U}^r$ extends both $\nice{U}^p$ and $\nice{U}^q$
by Definition \ref{def:P}(b).
Since $\nice{U}^r$ extends $\nice{U}^q$, by
condition (iii$_{\nice{U}^{r}}^{\nice{U}^{q}}$) of Definition \ref{SPSb}, the equality $U_{n+1}^r \cap F(X^q) = U_{n+1}^q$
holds, implying that $h \in U_{n+1}^r$. 
Similarly, by
condition (iii$_{\nice{U}^{r}}^{\nice{U}^{q}}$) of Definition \ref{SPSb},
we have $n+1\le n^q\le n^r$. 
Since $\nice{U}^r$ is a \fns\ for $F(X^r)$ by 
condition (3$_r$) of 
Definition \ref{def:P}(a),
we have
$e \in U_{n+1}^r$ 
by Remark \ref{identity:remark}.
Since $p\in B_{\{x\}}$, we have $x\in X^p$ by definition of $B_{\{x\}}$.
Since $\nice{U}^r$ extends $\nice{U}^p$, we have 
$X^p\subseteq X^r$ by
condition (i$_{\nice{U}^{r}}^{\nice{U}^{p}}$) of Definition \ref{SPSb}.
Therefore, $x\in X^r$, and so 
$y =x^{\varepsilon}\in \overline{X^r}$ by Definition \ref{def:bar}.
Since $\nice{U}^r$ is a \fns\ for $F(X^r)$,
by applying condition (3$_{\nice{U}^r}$) of Definition \ref{SPSa}, we get that 
\begin{equation} 
\label{eq:57:k}
y \cdot h \cdot y^{-1} 
=
y \cdot h \cdot e\cdot y^{-1} 
\subseteq 
y \cdot U_{n+1}^r \cdot U_{n+1}^r \cdot y^{-1} 
\subseteq 
\bigcup_{z\in \bar{Z}}z \cdot U_{n+1}^r \cdot U_{n+1}^r \cdot z^{-1} 
\subseteq U_{n}^r.
\end{equation}
Since $r \in \F$ and $n \leq n^r$, we have 
$U_{n}^r\subseteq U_n$ by \eqref{def:U_i}.
Combining this with \eqref{eq:57:k}, we conclude that
 $y \cdot h \cdot y^{-1} \in U_n$. 
 Since $h \in U_{n+1}$ was chosen arbitrarily, 
 this proves the inclusion
$y \fprod U_{n+1} \fprod y^{-1} \subseteq U_n$.
\end{proof}

\begin{claim}
\label{base:conjugation}
For every $n \in \N$ and each $g \in F(X)$, there exists $k \in \N$ such that $g \fprod U_k \fprod g^{-1} \subseteq U_n$.
\end{claim}

\begin{proof}
Let $n\in \N$ and $g \in F(X)$ be arbitrary. 
Then $g=y_1\fprod y_2\fprod  \dots\fprod y_m$ for some
$y_1, y_2,\dots, y_m \in \dash{X}$. 
Note that
$$y_m \fprod U_{n+m} \fprod y_m^{-1} \subseteq U_{n+m-1}$$
by Claim \ref{untangle}.
Applying this claim once again,
we get
$$
y_{m-1}\fprod (y_m \fprod U_{n+m} \fprod y_m^{-1}) \fprod y_{m-1}^{-1}
\subseteq y_{m-1}\fprod U_{n+m-1}\fprod y_{m-1}^{-1}
\subseteq U_{n+m-2}.
$$
By inductively applying Claim \ref{untangle} finitely many times,
we obtain the inclusion
$$
g\fprod U_{n+m} \fprod g^{-1}
=
y_1\fprod \dots \fprod y_m\fprod U_{n+m} \fprod y_m^{-1}\fprod\dots\fprod y_1^{-1}
\subseteq U_n.
$$
Therefore, it suffices to let $k = n+m$.
\end{proof}

\begin{claim}
\label{claim:9}
The family 
$\B$ as in \eqref{eq:fam:B}
is a neighbourhood base at $\e{F}$ 
of some Hausdorff group topology $\mathscr{T}$ on $F(X)$.
\end{claim}
\begin{proof}
It easily follows from Claims \ref{claim:6} and \ref{claim:7}
that $U_m\subseteq U_n$ whenever $n,m \in \N$ and $n \leq m$. 
Combined with \eqref{eq:fam:B}, this implies that $\B$ is a filter base. By \eqref{eq:fam:B}, Claim \ref{claim:7} and Claim \ref{base:conjugation}, 
$\B$ has the following three properties.

\begin{itemize}
\item For every $U \in \B$, there exists $V \in \B$ such that $V \fprod V \subseteq U$;
\item For every $U \in \B$, there exists $V \in \B$ such that $V^{-1} \subseteq U$;
\item For every $U \in \B$ and each $g \in F(X)$, there exists $V \in \B$ such that $g \fprod V \fprod g^{-1} \subseteq U$.
\end{itemize}

By \cite[Theorem 3.1.5]{1}, the family 
$$
\mathscr{T} = \{ O \subseteq F(X) : \forall\ g \in O\ \exists\ U \in \B\ (gU \subseteq O) \}
$$ 
is a topology on the free group $F(X)$ making it into a topological group such that the family $\B$ is a neighbourhood base at $\e{F}$ comprised of $\mathscr{T}$-neighborhoods of $\e{F}$. 
It follows from 
Claim \ref{claim:6} and \cite[Theorem 4.1.1]{1} that
$\mathscr{T}$
is Hausdorff. 
\end{proof}

\begin{claim}
The topological group $(F(X),\mathscr{T})$ has the algebraic small subgroup generating property. 
\end{claim}
\begin{proof}
We are going to check that the topological group $(F(X),\mathscr{T})$ satisfies Definition \ref{def:ASSGP}.

Let 
$W$ be a neighbourhood of $e$ in $(F(X),\mathscr{T})$.
By Claim \ref{claim:9},
there exists $n\in\N$ such that $U_n\subseteq W$. 

Fix 
$g \in F(X)\setminus\{e\}$ and recall that $A_n \cap D_g \in \mathscr{D}$. By the choice of $\F$, there exists some $q \in A_n \cap D_g \cap \F$. Since $q \in A_n$, the inequality $n \leq n^q$ holds by definition of $A_n$, and therefore  \eqref{def:U_i} tells us that \begin{equation}
\label{eq:57:l}
U_{n^q}^q \subseteq U_n^q \subseteq U_n \subseteq W. \end{equation}
Finally, since $q \in D_g$, we get $g \in \grp{\Cyc(U_{n^q}^q)}$. From \eqref{eq:57:l} we obtain that 
\begin{equation}
\label{eq:59:i}
g \in \grp{\Cyc(U_{n^q}^q)} \subseteq \grp{\Cyc(U_{n})} \subseteq \grp{\Cyc(W)}.
\end{equation}
Finally, $\grp{e}=e\in W$ implies that 
$e\in \Cyc(W)\subseteq \grp{\Cyc(W)}$.
Since \eqref{eq:59:i} holds for every $g \in F(X)\setminus\{e\}$,
we have proved that $F(X) \subseteq \grp{\Cyc(W)}$.
Since the converse inclusion clearly holds, 
we proved that $F(X)=\grp{\Cyc(W)}$.
Since this equality holds
  for every neighbourhood $W$ of $e$ in $(F(X),\mathscr{T})$,
we conclude that $(F(X),\mathscr{T})$ has the algebraic small subgroup generating property.
\end{proof}

Since $(F(X),\mathscr{T})$ is Hausdorff and has a countable base at $\e{F}$ by Claim \ref{claim:9}, it is metrizable.
This concludes the proof of Theorem \ref{countable:theorem}.

\section{Proof of Theorem \ref{the:theorem}}
\label{sec:8}

\def\E{\mathbb{E}}
\def\supp{\mathrm{supp}}

Fix an infinite cardinal $\kappa$.
We are going to show that the free group with $\kappa$ many generators admits an $\assgp$ group topology.

Let $H$ be the free group with a countably infinite set of generators equipped with the $\assgp$ group topology constructed in Theorem \ref{countable:theorem}.
Let $\B$ be a countable base of $H$ at $e$ consisting of symmetric neighbourhoods of $e$.

Define 
$\E=[\kappa\times\N]^{<\omega}\setminus\{\emptyset\}$.
For every $E\in \E$
fix an injection $\theta_E: E\to H$ such that $\theta_E(E)$ is an independent subset of $H$. (It suffices to send $E$  injectively into a subset of the infinite set of generators of $H$.)

Since 
$\kappa$ is infinite,
$|\E|=\kappa$ and 
$|[\E]^{<\omega}|=\kappa$. 
Since $H$ and $\B$ are countable,
we can fix a listing
$\{(\mathscr{E}_\alpha, h_\alpha, B_\alpha):\alpha<\kappa\}$
of triples $(\mathscr{E}_\alpha, h_\alpha, B_\alpha)$, where 
$\mathscr{E}_\alpha\in [\E]^{<\omega}\setminus\{\emptyset\}$,
$h_\alpha\in H^{\mathscr{E}_\alpha}$
and
$B_\alpha\in\B$ for every $\alpha<\kappa$,
having the following property:
\smallskip
\begin{itemize}
\item[($\diamond$)]
If $\mathscr{E}\in [\E]^{<\omega}\setminus\{\emptyset\}$,
$h\in H^{\mathscr{E}}$
and
$B\in\B$,
then
$(\mathscr{E}_\gamma,h_\gamma, B_\gamma)=(\mathscr{E}, h, B)$
for some $\gamma<\kappa$.
\end{itemize}
\smallskip

By transfinite induction on $\alpha<\kappa$, we shall define
$x_{\alpha,n}\in H^\E$ for every $n\in\N$ satisfying two properties:

\begin{itemize}
\item[(i$_\alpha$)]
If $(\alpha,n)\in E\in\E$ for some $n\in\N$, then $x_{\alpha,n}(E)=\theta_E(\alpha,n)$.
\item[(ii$_\alpha$)] 
There exist $n_0, k\in\N$ 
such that 
$x_{\alpha,n_0+i}\in\Cyc(W_\alpha)$  for $i=0,\dots,k$ 
and
$h_\alpha(E)=\prod_{i=0}^k x_{\alpha,n_0+i}(E)$ for every 
$E\in \mathscr{E}_\alpha$,
where $W_\alpha=B_\alpha^{\mathscr{E}_\alpha}\times H^{\E\setminus\mathscr{E}_\alpha}$.
\end{itemize}

Suppose that $\alpha<\kappa$ is an ordinal and $x_{\beta,n}\in H^\E$ were already defined for all $\beta<\kappa$ and $n\in\N$
in such a way that properties (i$_\beta$) and (ii$_\beta$) hold.
We shall  define $x_{\alpha,n}\in H^\E$ for every $n\in\N$ satisfying (i$_\alpha$) and (ii$_\alpha$).

Recall that $\mathscr{E}_\alpha\in [\E]^{<\omega}\setminus\{\emptyset\}$, so 
$\mathscr{E}_\alpha$ is a (non-empty) finite subset of $\E=[\kappa\times\N]^{<\omega}\setminus\{\emptyset\}$.
Therefore, $\bigcup \mathscr{E}_\alpha$ is a finite subset of $\kappa\times\N$, so we can fix $n_0\in \N$ such that 
\begin{equation}
\label{eq:not:in:union}
(\alpha,n)\not\in \bigcup \mathscr{E}_\alpha
\text{ whenever }
n\in\N
\text{ and }
n\ge n_0.
\end{equation}

Let $\mathscr{E}_\alpha=\{E_1,\dots,E_l\}$ be a faithfully indexed enumeration, where $E_j\in\E$ for every $j=1,\dots,l$.

Let $j=1,\dots,l$
be arbitrary. Since
$h_\alpha\in H^{\mathscr{E}_\alpha}$, we have
$h_\alpha(E_j)\in H$.
Since $B_\alpha\in\B$, it is an open neighbourhood of $e$ in $H$.
Since $H$ is $\assgp$, we can choose $k_j\in\N^+$
 and $g_0^j,\dots,g_{k_j}^j\in\Cyc(B_\alpha)$
such that
$$
h_\alpha(E_j)=\prod_{i=0}^{k_j} g_i^j.
$$

Define $k=\max\{k_j:j=0,\dots,l\}$.

Let $j=0,\dots,l$ be arbitrary.
For every $i$ with $k_j<i\le k$, define $g_i^j=e$.
Then
\begin{equation}
\label{eq:g_i^j}
g_0^j,\dots,g_{k}^j\in\Cyc(B_\alpha)
\ \ \text{ and }\ \ 
h_\alpha(E_j)=\prod_{i=0}^{k} g_i^j.
\end{equation}

For $n\in\N$ satisfying $n_0\le n\le n_0+k$, define $x_{\alpha,n}\in H^\E$ by
\begin{equation}
\label{triple:x}
x_{\alpha,n}(E)=
\begin{cases}
g_{n-n_0}^j &\text{if } E=E_j\text{ for some } j=1,\dots,l\\
\theta_E(\alpha,n) &\text{if } (\alpha,n)\in E\\
e & \text{otherwise}
 \end{cases}
\hskip30pt
\text{ for } E\in \E.
\end{equation} 
It should be noted that the first and the second line of 
 \eqref{triple:x} do not contradict each other. Indeed,
if $E=E_j$ for some $j=1,\dots,l$, then $E\in\mathscr{E}_\alpha$.
Since $n\ge n_0$, we have $(\alpha,n)\not\in E$ by 
\eqref{eq:not:in:union}.

For $n\in\N$ satisfying either $n<n_0$ or $n> n_0+k$, define $x_{\alpha,n}\in H^\E$ by
\begin{equation}
\label{double:x}
x_{\alpha,n}(E)=
\begin{cases}
\theta_E(\alpha,n) &\text{if } (\alpha,n)\in E\\
e & \text{otherwise}
 \end{cases}
\hskip30pt
\text{ for } E\in \E.
\end{equation}

Condition (i$_\alpha$) is satisfied by \eqref{triple:x} and \eqref{double:x}.

Let is check the condition (ii$_\alpha$). Let $E\in\mathscr{E}_\alpha$ be arbitrary. Then $E=E_j$ for some $j=1,\dots,l$.
It follows from \eqref{triple:x} that
$x_{\alpha,n_0+i}(E)=g_{i}^j$ for all $i=0,\dots,k$, so
\begin{equation}
\label{eq:x_alpha,n_0}
x_{\alpha,n_0+i}(E)\in\Cyc(B_\alpha)
\ \text{ for }\ 
i=0,\dots,k
\end{equation}
by \eqref{eq:g_i^j}
Furthermore,
\eqref{eq:g_i^j} also implies that
$$
h_\alpha(E)=h_\alpha(E_j)=
\prod_{i=0}^k g_i^j
=
\prod_{i=1}^k x_{\alpha, n_0+i}.
$$

Since $W_\alpha=B_\alpha^{\mathscr{E}_\alpha}\times H^{\E\setminus\mathscr{E}_\alpha}$
and \eqref{eq:x_alpha,n_0} holds for every $E\in\mathscr{E}_\alpha$,
we conclude that 
$x_{\alpha,n_0+i}\in \Cyc(W_\alpha)$
for every $i=0,\dots,k$.
This finishes the check of condition  (ii$_\alpha$). 

The inductive construction is complete.

\begin{claim}
\label{claim:X:independent}
$X=\{x_{\alpha,n}:\alpha<\kappa,n\in\N\}$ is an independent subset of $H^\E$.
\end{claim}
\begin{proof}
By \cite[Lemma 2.3]{Memoirs}, it suffices to show that every non-empty finite subset $Y$ of $X$ is independent in $H^\E$.
There exists a non-empty finite set $E\subseteq \kappa\times\N$ such that 
$Y=\{x_{\alpha,n}:(\alpha,n)\in E\}$.
Let 
$\eta:E\to Y$ be a surjective map defined by $\eta(\alpha,n)=x_{\alpha,n}$
for 
$(\alpha,n)\in E$.
Note that $E\in\E$ by our definition of $\E$.
Let $\pi_E: H^\E\to H$ be the projection on $E$'th coordinate.
If $(\alpha,n)\in E$, then 
$\pi_E(\eta(\alpha,n))=\pi_E(x_{\alpha,n})=x_{\alpha,n}(E)=\theta_E(\alpha,n)$
by (i$_\alpha$).
This shows that 
$\pi_E\restriction_Y\circ\eta=\theta_E$.
Since $\theta_E$ is an injection, so is $\eta$.
Since $\eta$ is also surjective, $\eta$ is a bijection between 
$E$ and $Y$, so it has the inverse map $\eta^{-1}: Y\to E$.
Now
$\pi_E\restriction_Y\circ\eta=\theta_E$ implies
$\pi_E\restriction_Y=\theta_E\circ \eta^{-1}$.
Since both $\eta^{-1}$ and $\theta_E$ are injections, so is $\pi_E\restriction_Y$.
Finally, it follows from our choice of $\theta_E$ that 
$\pi_E\restriction_Y(Y)=\theta_E\circ \eta^{-1}(Y)=\theta_E(\eta^{-1}(Y))=\theta_E(E)$ is an independent subset of $H$. It follows from \cite[Lemma 2.4]{Memoirs} that $Y$ is an independent subset of $H^\E$.
\end{proof}
\begin{claim}
\label{X:generates:ssgp:group}
The subgroup $G=\grp{X}$ of $H^\E$ generated by $X$ is $\assgp$.
\end{claim}
\begin{proof}
Let $U$ be an open neighbourhood of $e$ in $H^\E$.
In order to prove that $G$ is $\assgp$, it suffices to show that the equality $\grp{\Cyc(U\cap G)} = G$ holds.

Let $g \in G$ be arbitrary. By definition there exists $l \in \N$, $\alpha_i \in \kappa$ and $n_i \in \N$ for every $l = 1, \dots, l$ such that
\begin{equation}
g = \prod_{i = 1}^{l} x_{\alpha_i, n_i}^{\varepsilon_i}
\end{equation}
for some $\varepsilon_i = 0,1$ for every $i = 1, \dots, l$.
\end{proof}
Let \begin{equation*}
\mathscr{E}_1 = \{(\alpha_i,n_i): i = 1,\dots,l\}.
\end{equation*}
Now, since $U$ is open in $H^\E$ there exists a non-empty finite subset $\mathscr{E}_2$ of $\E$ and $ B \in \B$ such that $B^{\mathscr{E}_2}\times H^{\E \setminus \mathscr{E}_2} \subseteq U$. Let $\mathscr{E} = \mathscr{E}_1 \cup \mathscr{E}_2$, then
$\mathscr{E} \in [\E]^{<\omega}\setminus\{\emptyset\}$ and define $h=g\restriction_{\mathscr{E}}\in H^\mathscr{E}$.

By ($\diamond$), we can find $\alpha \in \kappa$  such that $( \mathscr{E}, h, B) = (\mathscr{E}_\alpha, h_\alpha, B_\alpha)$. Let $n_0 \in \N$ as in condition (ii$_\alpha$), then there exist $n_0,k \in \N$ such that $x_{\alpha,n_0+i}\in\Cyc(W_\alpha)$  for $i=0,\dots,k$
and
$h_\alpha(E) = h(E) =\prod_{i=0}^k x_{\alpha,n_0+i}(E)$ for every 
$E\in \mathscr{E}$,
where \begin{equation} \label{basic_set_open} W_\alpha=B_\alpha^{\mathscr{E}_\alpha}\times H^{\E\setminus\mathscr{E}_\alpha}=
B^{\mathscr{E}}\times H^{\E\setminus\mathscr{E}}
\subseteq 
B^{\mathscr{E}_2}\times H^{\E\setminus\mathscr{E}_2}
\subseteq U
\end{equation}

In particular, $\Cyc(W_\alpha)\subseteq \Cyc(U)$, which implies
$x_{\alpha,n_0+i}\in\Cyc(U)$  for $i=0,\dots,k$.
Since $x_{\alpha,n_0+i}\in X\subseteq G$, we have
$x_{\alpha,n_0+i}\in \Cyc(U\cap G)$  for $i=0,\dots,k$.

Define \begin{equation} g_1 = \prod_{i=0}^k x_{\alpha,n_0+i} \text{ and } g_2 = g \cdot g_1^{-1}.
\end{equation} 

By definition we have that $g_1,g_2 \in G$ since $G$ is a subgroup of $H^{\mathscr{E}}$. By the construction of $g_1$, we also know that $g_1 \restriction_\mathscr{E} = g \restriction_\mathscr{E}$.

Now, observe that $g_2 \in \Cyc(U)$; indeed, if $E \in \mathscr{E}$ then 
$$g_2(E) = g(E) \cdot (g_1(E))^{-1} = g(E) \cdot g(E)^{-1} = e_H.$$

Since $\{e_H\}^\mathscr{E} \times H^{\E\setminus \mathscr{E}}$ is a subgroup of $H^\mathscr{E}$, then by 
\eqref{basic_set_open} we have that \begin{equation} g_2 \in \{e_H\}^\mathscr{E} \times H^{\E\setminus \mathscr{E}} \subseteq \Cyc(U).
\end{equation}
This shows that $g_2 \in \Cyc(U \cap G)$,  and so \begin{equation} g = g_2 \cdot g_1 = g_2 \cdot \prod_{i=0}^k x_{\alpha,n_0+i}
\end{equation}
where for every $i = 0, \dots, k$ we have that $x_{\alpha,n_0+i}, g_2 \in \Cyc(U \cap G)$.
Thus, $g\in \grp{\Cyc(U\cap G)}$.
Since this holds for an arbitrary $g\in G$, we conclude that
$\grp{\Cyc(U\cap G)} = G$.

It follows from Claim \ref{claim:X:independent} that $G$ is a free group with generating set $X$.
Since $\kappa$ is an infinite cardinal, $|X|=|\kappa\times\N|=\kappa$. 
By Claim \ref{X:generates:ssgp:group}, the subspace topology $G$ inherits from $H^\E$ is $\assgp$.

\section{Questions}
\label{sec:questions}

The free group with one generator is isomorphic to the group $\Z$ of integer numbers, so it does not admit an $\ssgp$ group topology by \cite[Corollary 3.14]{CG},
and therefore it also cannot have an $\assgp$ group topology.
In view of this remark, 
Theorem \ref{the:theorem}
motivates
the following question.

\begin{question}
Let $n \in \N$ with $n \ge 2$. Can the free group with $n$ generators admit either an $\assgp$ or an $\ssgp$ group topology? 
\end{question}

Comparison of Theorems \ref{countable:theorem} and \ref{the:theorem} suggests the following question:

\begin{question}
Can the $\assgp$ group topology in Theorem \ref{the:theorem} be chosen to be metric?
\end{question}

In fact, a more general questions seems quite intriguing.

\begin{question}
If a group $G$ admits an $\ssgp$ group topology, must $G$ also admit a metric $\ssgp$ group topology? What if $G$ is abelian? 
\end{question}

The $\assgp$ version of this question also makes sense.

\begin{question}
If a group $G$ admits an $\assgp$ group topology, must $G$ also admit a metric $\assgp$ group topology? What if $G$ is abelian? 
\end{question}

The following problem may be considered as a ``heir'' of Question
\ref{ques1a}:

\begin{problem}
Describe the algebraic structure of 
 (abelian) groups which admit an $\assgp$ group topology.
\end{problem}

The authors made a substantial progress on this problem 
in \cite{SY-assgp}.

\end{document}